\documentclass[oneside]{amsart}
\usepackage{amsfonts,amsthm,amsmath}
\usepackage{amssymb}
\usepackage{amscd}
\usepackage[utf8]{inputenc}
\usepackage[a4paper]{geometry}
\usepackage{enumerate}
\usepackage[usenames,dvipsnames]{color}
\usepackage{soul}
\usepackage{array}
\usepackage{array,tabularx}
\usepackage{xypic} 
\usepackage{epic}
\usepackage{float}
\usepackage{graphicx}
\usepackage{caption}
\usepackage{subcaption}
\usepackage[normalem]{ulem}

\usepackage{tikz}
\usetikzlibrary{arrows, intersections, calc, matrix}

\numberwithin{equation}{section}
\numberwithin{equation}{subsection}

\theoremstyle{plain}
\newtheorem{theorem}[equation]{Theorem}
\newtheorem{lemma}[equation]{Lemma}
\newtheorem{proposition}[equation]{Proposition}

\newtheorem{question}[equation]{Question}

\newtheorem{thm}[equation]{Theorem}
\newtheorem{cor}[equation]{Corollary}

\newtheorem{prop}[equation]{Proposition}

\newtheorem*{chtheorem}{Characterization Theorem}
\newtheorem*{chproblem}{Characterization Problem}
\newtheorem*{repproblem}{Representability Problem}

\theoremstyle{definition}
\newtheorem{example}[equation]{Example}
\newtheorem{remark}[equation]{Remark}

\newtheorem{definition}[equation]{Definition}

\DeclareMathOperator{\Hom}{{\rm Hom}}

\newcommand{\tX}{\widetilde{X}}

\def\C{\mathbb C}
\def\Q{\mathbb Q}

\def\Z{\mathbb Z}

\newcommand{\calv}{{\mathcal V}}

\newcommand{\calO}{{\mathcal O}}

\newcommand{\calS}{{\mathcal S}}

\newcommand{\V}{\calv}

\newcommand{\w}{\omega}
\newcommand{\al}{\alpha}
\renewcommand{\Z}{\mathbb{Z}}
\newcommand{\Na}{\mathbb{N}}

\newcommand{\Se}{\mathcal{S}}
\renewcommand{\Q}{\mathbb{Q}}
 \newcommand{\G}{\Gamma}
\newcommand{\oo}{\mathfrak{o}}

\newcommand{\q}{q}

\newcommand{\z}[1]{\widehat{#1}}
\newcommand{\h}[1]{\left\lbrace{#1}\right\rbrace}

\author{Zsolt Baja}
\thanks{}
\address{Babe\c{s}-Bolyai University, Faculty of Mathematics and Computer Science,\newline \hspace*{6mm} 
Str. Mihail Kog\u{a}lniceanu nr. 1, 400084 Cluj-Napoca, Romania}
\email{zsolt.baja@ubbcluj.ro}

\author{Tam\'as L\'aszl\'o}
\thanks{}
\address{Babe\c{s}-Bolyai University, Faculty of Mathematics and Computer Science,\newline \hspace*{6mm} 
Str. Mihail Kog\u{a}lniceanu nr. 1, 400084 Cluj-Napoca, Romania;\vspace*{1mm}\newline\hspace*{3mm} 
Alfr\'ed R\'enyi Institute of Mathematics, Re\'altanoda u. 13-15, H-1053 Budapest, Hungary}
\email{tamas.laszlo@ubbcluj.ro}
\thanks{TL is partially supported by the  J\'anos Bolyai Research Scholarship of the Hungarian Academy of Sciences and the NKFIH Grant ``\'Elvonal (Frontier)'' KKP 144148. He is very grateful to the `Erd\H os Center' for the warm hospitality and for providing with him an excellent research environment during his stay in Budapest.}

\title{Flat semigroups and weighted homogeneous surface singularities}

\begin{document}



\begin{abstract}
We consider numerical semigroups associated with normal weighted homogeneous surface singularities with rational homology sphere links. We say that a semigroup is representable if it can be realized in this way.

In this article we characterize the representable semigroups by proving that they are exactly those semigroups which can be written as a quotient of a flat semigroup. As an intermediate step, we study the class of flat semigroups and show that they can be represented by a special subclass of isolated complete intersection singularities whose defining equations can be given explicitely.
\end{abstract}

\maketitle


\linespread{1.2}


\pagestyle{myheadings} \markboth{{\normalsize Zs. Baja and T. L\'aszl\'o}} {{\normalsize Flat semigroups and weighted homogeneous surface singularities}}


\section{Introduction}

\subsection{} 

One of the most classical problems in the theory of numerical semigroups is the {\it Diophantine Frobenius problem}, asking the following: given the numerical semigroup $G(a_1,\dots,a_n)$ generated by relatively prime integers $a_1,\dots,a_n$, find the largest integer - called the {\it Frobenius number} - that is not contained in the semigroup, ie. it is not representable as a non-negative integer combination of $a_1,\dots,a_n$. 

Several ideas from different areas of mathematics have been studied to find `closed' formulae and algorithms to calculate the Frobenius number. From the point of view of the formulae, although several results for peculiar cases and general bound estimates exist in the literature, the problem is still open in full generality. In our forthcomming discussions we will only touch a small part of the `classical approaches', nevertheless the interested reader might consult for more details the excelent monograph of Ram\'irez Alfons\'in \cite{RamAlf}.

\subsection{} The present article focuses on a relatively new method developed by A. N\'emethi and the second author in \cite{stronglyflat}. This is based on a subtle connection with the theory of complex normal surface singularities. Using geometrical and topological techniques of singularity theory, one can solve the Frobenius problem for numerical semigroups which can be related with certain singularities. 

\vspace{0.2cm}
In the theory of surface singularities numerical semigroups appear naturally in many diferent contexts. In this article we consider the case of weighted homogeneous normal surface singularities. 

Let $(X,0)$ be a normal weighted homogeneous surface singularity. Being the germ of an affine variety $X$ with good $\mathbb{C}^*$-action, its affine coordinate ring is $\Z_{\geq 0}$-graded: $R_X=\oplus_{\ell\geq 0}R_{X,\ell}$. Then the set $$\calS_{(X,0)}:=\{\ell \in \Z_{\geq 0} \ | \ R_{X,\ell}\neq 0\}$$ is a numerical semigroup. In fact, by Pinkham \cite{art_pink} we know that if the link $M$ of the singularity - which is a negative definite Seifert $3$-manifold in this case - is a rational homology sphere, then $\calS_{(X,0)}$ is a topological invariant, associated with the link $M$, or with star-shaped minimal good dual resolution graph $\Gamma$. Accordingly, it will be also denoted by $\calS_{\Gamma}$, see section \ref{ss:reprsem}. 

In this note we will say that a numerical semigroup is {\it representable} if it can be realized as $\calS_{\Gamma}$.  

One also follows by \cite{art_pink} that the representable semigroups can be expressed as $\calS_{\Gamma}=\{\ell \in \Z \ | \ N(\ell)\neq 0\}$ where $N(\ell)$ is a quasi-linear function defined by the Seifert invariants, cf. section \ref{ss:WHS}. Using this presentation and the arithmetics of the star-shaped minimal good dual resolution graph associated with the singularity, \cite{stronglyflat} developed a `closed' formula for the Frobenius number. Moreover, in that work the authors proved that a special class of representable semigroups appears naturally in the classical theory of numerical semigroups: the semigroups associated with Seifert integral homology $3$-spheres are exactly the `strongly flat semigroups' with at least $3$ generators, considered by the work of Raczunas and Chrz\c astowski-Wachtel \cite{art_rch} as special semigroups realizing a sharp upper bound for the Frobenius problem. Note that the case of two generators will be threated in this article, clarifying completely the class of strongly flat semigroups.

\vspace{0.2cm}
The aforementioned results imposed the following natural questions towards to the study of representable semigroups (see \cite[sect. 8]{stronglyflat}): 
\begin{chproblem}
How can the representable semigroups be characterized?
\end{chproblem}

\begin{repproblem}
How big is the set of representable semigroups inside the set of all numerical semigroups? 
\end{repproblem}

\subsection{}The aim of this work is to give an answer to the characterization problem. In the sequel, we will provide an overview of our study and state the main result of the article. 

\vspace{0.2cm}
First of all, we observe that the quotient of a representable semigroup is also representable. Then we study the class of `flat semigroups' appeared in the classification theme of \cite{art_rch}. It turns out that they are crucial towards to the characterization problem. First, we prove that they are representable. Moreover, for a given presentation of a flat semigroup we construct a {\it canonical representative} $M$ whose specialty induces many of the properties of the semigroup. In particular, one can show that every flat semigroup is symmetric and its Frobenius number realizes a sharp upper bound considered by Brauer \cite{art_brauer1}. Furthermore, they are interesting from singularity theoretic point of view as well, since these semigroups can be associated with a special family of isolated complete intersection singularities whose equations are given explicitely in section \ref{s:ICIS}. The main theorem of this manuscript is as follows. 

\begin{chtheorem}
A numerical semigroup is representable if and only if it is a quotient of a flat semigroup.  
\end{chtheorem}

The strategy of the proof is as follows. From the previous observations one deduces that a quotient of a flat semigroup is representable. For the converse statement, we fix a negative definite Seifert rational homology sphere as a representative of the given numerical semigroup. Then, we prove that by perturbing the Seifert invariants in such a way that the associated semigroup does not change, we can always achieve a representative whose associated semigroup can be realized as a quotient of a flat semigroup.

\vspace{0.2cm}
It is worth to emphasize that the characterization theorem rephrases the representability problem completely via the language of numerical semigroup theory: 
\begin{center}
`{\it how big is the set of quotients of flat semigroups in the set of numerical semigroups?}' 
\end{center}
Further speculations regarding this question and a connection with the work of \cite{symmhalf} and \cite{symmquotient} with respect to quotients of symmetric semigroups can be found at the end of the article. 

\subsection{} The structure of the article is as follows. Section \ref{s:prliminaries} summarizes the necessary  preliminaries regarding the Frobenius problem, the `flat'  classification of numerical semigroups following  \cite{art_rch}, weighted homogeneous surface singularities and the arithmetics of their star-shaped minimal good dual resolution graphs. Furthermore, the last part define the representable semigroups and collects some already known results about them.   

Section \ref{s:rpsg} contains some new observations regarding representable semigroups. In the first part we introduce the sum of the SSR graphs (see section \ref{sec:op}) and using this concept we discuss the representability of semigroups with two generators. Furthermore, in section \ref{sec:bounds} we construct monoids as bounds for a representable semigroup. This serves the base idea for the study of flat semigroups in section \ref{sec:flat}. In this part we prove that they are representable (Theorem \ref{thm_flat_rep}) by constructing a canonical representative, and we study some of their properties. In particular, in section \ref{s:ICIS} we give explicit equations for a family of isolated complete intersection singularities whose links are the canonical representatives of a flat semigroup. 

Finally, section \ref{s:ch} explains a perturbation process for the Seifert invariants of a reresentative, and proves the main theorem (Theorem \ref{charthm}) of our work. Section \ref{ss:specrem} ends the paper by giving important examples and discussing a new reformulation of the representability problem.

%

\section{Preliminaries}\label{s:prliminaries}

\subsection{Numerical semigroups and their  `flat'  classification}

\subsubsection{\bf The Frobenius problem}

The Diophantine Frobenius problem aims to find an explicit formula for the greatest integer not representable as a nonnegative linear form of a given system of $d\geq 2$ relatively prime integers $1\leq a_1\leq \ldots \leq a_d$. The  integer defined in this way is called the {\it Frobenius number} of the system. In numerical semigroup language, let $G(a_1,\ldots,a_d)$ be the numerical semigroup (ie. submonoid of $\mathbb{N}$ with finite complement) generated by the integers from the above system. Then the Frobenius number is the largest gap of $G(a_1,\ldots,a_d)$, for which we will use the standard notation $f_{G(a_1,\ldots,a_d)}$.
 
The problem is still open in full generality, however several fomulae for special systems and general bounds exist in the literature. For example, the very first result is the  Sylvester formula expressing $f_{G(a_1,a_2)}=a_1a_2-a_1-a_2$.

For the classical approach and different aspects of the problem the interested reader might consult  the monograph of Ram\'irez Alfons\'in \cite{RamAlf}. Further details regarding the theory of numerical semigroups can be found eg. in \cite{book_nsgp} and \cite{book_nsgp2}.
\vspace{0.2cm}

In the sequel we will present some of the general bounds for the Frobenius number which will be important for our purpose. 
\vspace{0.2cm}

Brauer \cite{art_brauer1} considered the following upper bound 
\begin{equation}\label{eq:f=T}
 f_{G(a_1,\dots,a_n)}\leq T(a_1,\dots,a_n):=\sum_{i=1}^n \Big(\frac{d_{i-1}}{d_i}-1\Big) a_i, 
\end{equation}
where $d_0=0$ and $d_i=\gcd(a_1,\dots,a_i)$ for all $i\geq 1$.  Moreover, in \cite{art_brauer2} the authors characterized those semigroups which satisfy the equality in (\ref{eq:f=T}), namely 
$$ f=T\iff a_{i+1}/d_{i+1}\in G(a_1/d_i,\dots,a_{i}/d_{i})\,\ \ \mbox{for every} \ 1 \leq i\leq n-1.$$
We notice that the value of $T$, as well as the above criterion depends on the order of the generators, and in general only an appropriate permutation gives the equality $f=T$.

Raczunas and Chrz\c{a}stowski-Wachtel \cite{art_rch} 
characterized another subclass of semigroups for which $f=T$ holds and $T$ can be expressed  independently on the permutation of  the generators. These are the so-called {\it flat semigroups}. 
In particular, they considered another upper bound 
\begin{equation}\label{eq:B}
B(a_1,a_2,\dots,a_n):=(n-1)\mathrm{lcm}(a_1,\dots,a_n)-\sum_i a_i,
\end{equation}
which satisfies 
$ f_{G(a_1,\dots,a_n)}\le T(a_1,\dots,a_n)\le B(a_1,\dots,a_n)$,  
and characterized the class of semigroups for which the equality $f=T=B$ holds. They are called the {\it strongly flat semigroups} and form a subclass of the flat ones. \\
In the followings we will describe precisely these classes following the discussion in \cite{art_rch}.

\subsubsection{\bf Flat classification of semigroups  \cite{art_rch}}\label{ss:classsgp}
Based on the decomposition of the generators, one considers four  `shades' 
of flatness: {\it strongly flat, flat, almost flat and non-flat semigroups}.

Let $A=\{a_1,\dots,a_n\}$ be a system of generators of a 
numerical semigroup $\calS$, $\gcd(a_1,\dots,a_n) =1$. If we consider the numbers $\q_i:=\gcd(a_1,\dots,a_{i-1},a_{i+1},\dots, a_n)$ and $\widehat{\q}_i:=\prod_{j\neq i} \q_j$ for all $i\in\{1,\dots,n\}$, then one follows that $\gcd(\q_i,\q_j)=1$ for every $ i\neq j$. Hence $\widehat{\q}_i \ | \ a_i$ and we can define $\widehat{s}_i:=a_i/\widehat{\q}_i$ (note also that $\gcd(\widehat{s}_i,\q_i)=1$). Then the system of generators can be presented in the following form:
\begin{equation}\label{eq:genset}
 A=\{a_1,\dots,a_n\}=\{\widehat{s}_1 \widehat{\q}_1,\dots, \widehat{s}_n \widehat{\q}_n\}.
\end{equation}

\begin{definition}
    \label{df:flatnes}
The set $A$ is 
   \begin{itemize}
       \item {\it Strongly flat} {\bf (SF)} if  one has $\widehat{s}_i=1$ for all $i$;
       \item {\it Flat} {\bf (F)}  if there exists an $i$ such that $\widehat{s}_i=1$;
       \item {\it Almost flat} { \bf (AF)} if there exists an $i$ such that $\q_i>1$ and 
       \item {\it Non-flat} {\bf (NF)} if  for all $i$ one has $\q_i=1$.
   \end{itemize} 
We say that a numerical semigroup $\calS$ is strongly flat, flat, almost flat or non-flat if the corresponding condition is satisfied for the minimal set of generators. 
\end{definition}
\begin{remark}
Note that the full semigroup $\calS=\mathbb{N}$ and the semigroups minimally generated by two elements are automatically {\bf SF}. On the other hand, we have $\textbf{SF}\subset \textbf{F}\subset \textbf{AF}$. Moreover, if one of these three conditions is satisfied for a non-minimal set of generators, then it is automatically satisfied for the minimal too.  This property does not hold for $\textbf{NF}$. 
\end{remark}

As we have already mentioned, the strongly flat semigroups are characterized by the equality $f_{\calS}=B(a_1,\dots,a_n)$, where  $\{a_1,\dots,a_n\}$ is the minimal set of generators of $\calS$. 
Moreover, using the presentation (\ref{eq:genset}) of the generators 
and the notation $a:=\mathrm{lcm}(a_1,\dots,a_n)$, in this case the Frobenius number can be rewritten as 
$$
f_{\calS}=a\Big(n-1-\sum_i \frac{1}{\q_i}\Big).    
$$

If $\calS$ is flat then $f_{\calS}=T(b_1,\dots,b_n),$ where $b_1,\dots,b_n$ is 
an appropriate permutation of $a_1,\dots,a_n$. However, in this case the Frobenius number can also be expressed in a direct form as 
\begin{equation}\label{eq:Frobflat}
f_{\calS}=\sum_i\left(q_i-1\right)a_i-\prod_i q_i,    
\end{equation}
see \cite[Thm. 2.5]{art_rch}.
\vskip 0.3cm

\subsection{On weighted homogeneous surface singularities}\label{ss:WHS}
\subsubsection{} A normal weighted homogeneous surface singularity $(X,0)$ is defined as the germ at the origin of an affine surface $X$ with a  good and effective  $\mathbb{C}^*$-action. This  means that its affine coordinate ring is $\mathbb{Z}_{\geq 0}$--graded: $R_X=\oplus_{\ell\geq 0} R_{X,\ell}$. (In fact, all finitely generated $\mathbb{Z}_{\geq 0}$--graded $\mathbb{C}$-algebra corresponds to an affine variety with good $\mathbb{C}^*$-action.)

Let $(X,0)$ be a normal weighted homogeneous surface singularity. Then $E_0:=(X\setminus\{0\})/\mathbb{C}^*$ is a smooth compact curve. If we denote by $T$ the closure of the graph of the map $X\setminus \{0\}\to E_0$ in $X\times E_0$, then the first projection $T\to X$ is a modification of $(X,0)$, while the second projection $T\to E_0$ is a Seifert line bundle with zero section $E_0$. $T$ has only cyclic quotient singularities at the intersection of $E_0$ with each singular fiber. After resolving these singularities we get the minimal good resolution of $\pi:\widetilde X\to X$ whose dual resolution graph $\Gamma$ is \textit{star-shaped}, cf. \cite{OW}. Thus, the exceptional divisor $\pi^{-1}(0)$ is a normal crossing divisor and only the central curve $E_0$ can have self-intersection number $-1$. We denote by $\{E_v\}_{v\in \calv}$ the irreducible components of $\pi^{-1}(0)$.

On the other hand the $\mathbb{C}^*$-action induces an $S^1$--Seifert action on the link $M$ of the singularity. In particular, $M$ is a negative definite Seifert 3-manifold characterized by its normalized Seifert invariants which will be denoted by $Sf=(-b_0,g;(\alpha_i,\omega_i)_{i=1}^d).$

\subsubsection{\bf Combinatorics and lattices associated with the resolution graph}\label{ss:seifert}

Let $\Gamma$ be the star-shaped minimal good dual resolution graph of a   normal weighted homogeneous surface singularity $(X,0)$, or, equivalently, a plumbing representation of the link $M$. Then $\Gamma$ has $d$ legs with $d\geq 3$. (A leg is a chain of vertices which is connected to the central vertex). Each leg is determined by the so-called normalized Seifert invariant $(\alpha_i,\omega_i)$,
where $0<\omega_i <\alpha_i$ and $\gcd(\alpha_i,\omega_i)=1$. The $i^{\mathrm{th}}$ leg has $\nu_i$ vertices, say $v_{i1},\ldots, v_{i\nu_i}$ ($v_{i1}$ is connected to the central vertex $v_0$)
with Euler decorations (self-intersection numbers)
$-b_{i1},\ldots, -b_{i\nu_i}$, which is given by the Hirzebruch--Jung (negative) continued fraction expansion
$$ \alpha_i/\omega_i=[b_{i1},\ldots, b_{i\nu_i}]=
b_{i1}-1/(b_{i2}-1/(\cdots -1/b_{i\nu_i})\dots) \ \  \ \ (b_{ij}\geq 2).$$
All these vertices (except $v_0$) have genus--decorations zero. The central vertex $v_0$ corresponds to the central genus $g$ curve $E_0$ with self-intersection number $-b_0$.
It is also useful to define $\omega_i'$ satisfying 
\begin{equation}\label{eq:w'}
\omega_i\omega_i'\equiv 1 \, (\mathrm{mod} \,\alpha_i), \ \ 0< \omega_i'<\alpha_i.
\end{equation}
One also knows that  $\alpha_i=\det(\Gamma_i)$, the determinant of the $i^{th}$--leg $\Gamma_i$, $\omega_i=\det(\Gamma_i\setminus v_{i1})$ and  $\omega_i'=\det(\Gamma_i\setminus v_{i\nu_i})$.

In the sequel we will assume that the link $M$ is a {\em (Seifert) rational homology sphere}, or, equivalently, the curve $E_0$ has $g=0$. In this case we will omit the genus from the notations and we will simply write $Sf=(-b_0;(\alpha_i,\omega_i)_{i=1}^d)$. 

The smooth complex analytic surface $\widetilde X$ is the plumbed $4$-manifold associated with $\Gamma$, with boundary $\partial \widetilde{X} = M$. 
We define the lattice $L$ as $H_2(\tilde X,\Z)$, endowed with the non-degenerate negative definite intersection form $I:=(\, , \,)$. It is freely generated by the (classes of the) exceptional divisors $E_v$, $v\in \calv$, that is,
$L=\oplus_{v\in \calv} \Z\langle E_v \rangle$. The dual lattice $L':= \Hom(H_2(\widetilde X,\Z),\Z)$ can be identified with $H_2(\widetilde{X}, M, \Z)$. Moreover, one has that $L'/L\cong H_1(M,\Z)$, which is a finite group and it will be denoted by $H$, cf. \cite{Nfive,nembook}.

Since the intersection form is non--degenerate, $L'$ embeds into $ L_{{\mathbb Q}}:=L\otimes {\mathbb Q}$,
and it can be identified with the rational cycles $\{l'\in L_{{\mathbb Q}}\,:\, (l',L)_{{\mathbb Q}}\subset \Z\}$, where $(\,,\,)_{{\mathbb Q}}$ denotes the extension of the intersection form to $L_{{\mathbb Q}}$.
 Hence, in the sequel we regard $L'$ as $\oplus_{v\in \V} \Z\langle E^*_v \rangle$, the lattice generated by the (anti-)dual cycles $E^*_v\in L_{{\Q}}$, $v\in \V$, where $(E_u^*,E_v)_{{\Q}}=-\delta_{u,v}$ (Kronecker delta) for any $u,v\in \calv$.

We can consider the anti-canonical cycle $Z_K\in L'$ defined by the adjunction formulae $(Z_K,E_v)=e_v+2$ for all $v$. 

We say that the singularity $(X,0)$, or its topological type, is \emph{numerically Gorenstein} if $Z_K\in L$. Note that the $Z_K\in L$ property is independent of the resolution, 
since $Z_K\in L$ if and only if the line bundle $\Omega^2_{X\setminus \{0\}}$ of holomorphic 2--forms
on $X\setminus \{0\}$ is topologically trivial. $(X,0)$ is called \emph{Gorenstein} if 
$\Omega^2_{\tX}$ (the sheaf of holomorphic 2--forms) is isomorphic to $ \calO_{\tX}(-Z_K)$ (or,
equivalently, if the line bundle $\Omega^2_{X\setminus \{0\}}$ is holomorphically trivial).
Note that the adjunction formulae imply the identity 
 \begin{equation}\label{ZK}
 Z_K-E=\sum_{v\in\mathcal{V}}(\delta_v-2)E^*_v,
 \end{equation}
where we denote $E:=\sum_{v\in\mathcal{V}}E_v$ and $\delta_v$ is the valency of the vertex $v$.

\subsubsection{\bf Some key numerical invariants}
The orbifold Euler number of the Seifert 3-manifold $M$ is defined as $e:=-b_0+\sum_i\omega_i/\alpha_i$. Then the negative definiteness of the intersection form is equivalent with  $e<0$. 

Let $\mathfrak{h}:=|H|$ be the order of $H=H_1(M,\mathbb{Z})=L'/L$, and
let  $\mathfrak{o}$ be the order of the class $[E^*_0]$ (or of
the generic $S^1$ Seifert--orbit) in $H$. 
Then,  using the notation  $\alpha:=\mathrm{lcm}(\alpha_1,\ldots,\alpha_d)$, one shows that (see eg. \cite{neumann.abel})
\begin{equation}\label{eq:sei2}
 \mathfrak{h}=\alpha_1\cdots\alpha_d|e| \ \ \mbox{and} \ \ \mathfrak{o}=\alpha|e|.
\end{equation}
In particular, if $M$ is an integral homology sphere (called Seifert homology sphere) then necessarily
all $\alpha_i$'s are pairwise relatively prime and by (\ref{eq:sei2}) $\alpha |e|=1$. 
This gives the Diophantine equation $(b_0-\sum_i\omega_i/\alpha_i)\alpha=1$, which determines all $\omega_i$ and $b_0$
 uniquely by the $\alpha_i$'s. The corresponding Seifert homology sphere is denoted by $\Sigma(\alpha_1,\dots,\alpha_d)$.

Next, we define the combinatorial number
\begin{equation}\label{eq:gamma}
\gamma:=\frac{1}{|e|}\cdot \Big( d-2-\sum_{i=1}^d \frac{1}{\alpha_i}\Big)\in \mathbb{Q},
\end{equation}
which has a central importance  regarding properties of
weighted homogeneous surface singularities or Seifert rational homology spheres.
It has several interpretations: it is the `exponent' of the weighted
homogeneous germ $(X,0)$; $-\gamma$ is also called the `log-discrepancy' of $E_0$; $\mathfrak{o}\gamma$ is usually named as the Goto--Watanabe $a$--invariant  of the universal abelian
cover of $(X,0)$, and $e\gamma$ appears as the orbifold Euler characteristic in \cite{neumann.abel} (see also
\cite[3.3.6]{art_nem_emb}). 

Nevertheless, the most important interpretation for our purpose will be  the following. In a star--shaped graph the $E_0$-coefficients of all
 $E^*_v$ associated with the end--vertices   are computed by $-(E_v^*,E^*_0)=1/(|e|\alpha_v)$ and the
 $E_0$-coefficient of
 $E^*_0$ is  $-(E_0^*,E^*_0)=1/|e|$
 (cf. \cite[(11.1)]{NOSZ}).
 Hence,  (\ref{ZK}) gives that the $E_0$-coefficient of $Z_K$ is exactly $\gamma+1$.

For any $i=1,\dots,d$ let us denote by $E_i$ the base element of the $i^{th}$ end--vertex $v_{i\nu_i}$ and compute the $E_i$-coefficient of $Z_K$. 
Using the identities $(E_i^*,E_j^*)=(e\alpha_i\alpha_j)^{-1}$ for $i\not=j$ and
$(E_i^*,E_i^*)=(e\alpha_i^2)^{-1}-\omega'_i/\alpha_i$ if $i=j$, 
 cf. \cite[(11.1)]{NOSZ}, 
 (\ref{ZK}) and a computation deduce that 
 \begin{equation}\label{eq:ZKend}
 -(Z_K,E^*_i)=1+(\gamma-\omega_i')/\alpha_i.
 \end{equation}
On the other hand, by \cite[Lemma 2.2.1]{LSz_module}  we know the  expression
$$E^*_{v_{ij}}=m_{ij}E^*_i-\sum_{j<r\leq \nu_i}m_{ijr}E_{v_{ir}},$$
where $m_{ij}$ and $m_{ijr}$ are positive integers. This yields that $-(Z_K,E^*_{v_{ij}})=M_{ij}(Z_K,-E^*_i)+M'_{ij}$ for some $M_{ij},M'_{ij}\in\mathbb{Z}$, which gives us the following observation.
\begin{lemma}\label{lem:ngor}
$\Gamma$ is numerically Gorenstein if and only if  $\gamma\in \mathbb{Z}$ and $\gamma\equiv\omega_i'\ ({\rm mod}\ \alpha_i)$ for all $i=1,\dots,d$.
\end{lemma}


\begin{remark}
Note that $\gamma|e|=d-2-\sum_i 1/\alpha_i$ is negative if and only if $\pi_1(M)$ is finite, cf.  \cite{CR}. This
can happen only if $d=3$ and $\sum_i1/\alpha_i>1$, and in this case $(X,0)$ is a quotient singularity,
hence rational. If $(X,0)$ is not rational then $\gamma\geq 0$, that is, the $E_0$--coefficient of $Z_K$ is
$\geq 1$. In fact, in these cases all the coefficients of $Z_K$ are strict positive, see eg. \cite[3.2.5]{stronglyflat}. Moreover, in the numerically Gorenstein non-rational case --- when we already know that $\gamma\geq 0$ --- by the congruence from Lemma \ref{lem:ngor} we get the stronger $\gamma\geq 1$.
\end{remark}


\subsection{Representable numerical semigroups}\label{ss:reprsem}

\subsubsection{\bf Numerical Semigroups associated with weighted homogeneous surface singularities \cite{stronglyflat}}
We define the set $\calS_{(X,0)}:=\{\ell\in\mathbb{Z}_{\geq 0} | R_{X,\ell}\neq 0\}$. It is a numerical semigroup by the grading property and it is called the {\it numerical semigroup associated with $(X,0)$}.

By  the work of \cite{art_pink} the complex structure of $(X,0)$ is completely recovered by the Seifert invariants and the configuration of points $\{P_i:=E_0\cap E_{i1}\}_{i=1}^d \subset E_0$, where  $E_0$ is the  central curve and $E_{i1}$ is the component corresponding to $v_{i1}$ in $\Gamma$. Furthermore, the graded ring of the local algebra of the singularity is interpreted by the so-called {\em Dolgachev--Pinkham--Demazure} formula as 
\begin{equation}\label{eq:DPD}
R_X=\oplus_{\ell\geq 0} R_{X,\ell}=\oplus_{\ell\geq 0}
H^0(E_0,\mathcal{O}_{E_0}(D^{(\ell)})),
\end{equation}
with $D^{(\ell)}:=\ell (-E_0|_{E_0})-\sum_{i=1}^d\lceil\ell\omega_i/\alpha_i\rceil P_i$, where $\lceil r\rceil$ denotes the smallest integer greater or equal to $r$.

In particular, when $M$ is a rational homology sphere, ie. $E_0\simeq \mathbb{P}^1$, (\ref{eq:DPD}) implies that $\dim(R_{X,\ell})=\max\{0,1+N(\ell)\}$ is topological, where $N(\ell)$
is  the quasi-linear function 
\begin{equation}\label{defN}
N(\ell):=\deg D^{(\ell)}=b_0\ell-\sum_{i=1}^d\Big\lceil \frac{\ell \omega_i}{\alpha_i}\Big\rceil.
\end{equation}
Since $-\lceil x\rceil \leq -x$ one obtains $N(\ell)\leq |e|\ell$,
hence  $N(\ell)<0$ for $\ell<0$.
This means that the semigroup $\calS_{(X,0)}$ can be described purely with the Seifert invariants  as follows
\begin{equation}\label{eq:sgptop}
\mathcal{S}_{(X,0)}=\{\ell\in\mathbb{Z} \ | \ N(\ell)\geq 0\}.
\end{equation}
Since in this case $\calS_{(X,0)}$ is a topological invariant, we will frequently use the notations $\calS_{M}$, or $\calS_{\Gamma}$ as well. 

\begin{definition}
 We say that a numerical semigroup $\Se$ is {\it representable} if it can be realized as the semigroup associated with a weighted homogeneous
    surface singularity $(X,0)$ with rational homology sphere link. 
 Accordingly, we will say that $(X,0)$, or its link $M$, or the graph $\G$
 is a {\it representative} of the numerical semigroup $\Se$.
\end{definition}

Finally, we list some important properties of the quasi-linear function $N(\ell)$ which will be used later in our discussion.

\begin{prop}{\cite{art_nem_emb}, \cite[Prop. 3.2.11 \& 3.2.13]{stronglyflat}}\label{prop:INEQ} \

(a) \ $-(\alpha-1)|e|-d\leq N(\ell)-\lceil \ell/\alpha\rceil \alpha |e| \leq -1$. In particular
$\lim_{\ell\to\infty}N(\ell)=\infty$.

(b) \ If $\ell>\gamma$ then $h^1(E_0,\mathcal{O}_{E_0}(D^{(\ell)}))=0$, ie. $N(\ell)\geq -1$.


(c) \ $N(\alpha)=\alpha(b_0-\sum_i\omega_i/\alpha_i)=\alpha |e|=\mathfrak{o}>0$.

(d) \ $N(\ell+\alpha)=N(\ell)+N(\alpha)=N(\ell)+\mathfrak{o}>N(\ell)$ for any $ \ell\geq 0$.

(e) \ $N(\ell)\geq 0$ for any $\ell>\alpha+\gamma$.

(f) \ If the graph is numerically Gorenstein  (that is, $Z_K\in L$), then
\begin{equation}\label{SHSsym}
N(\ell)+N(\gamma-\ell)=-2 \ \ \mbox{for any} \ \ell\in\mathbb{Z}.
\end{equation}
\end{prop}

\subsubsection{\bf The Frobenius number of representable semigroups}

Let $\Gamma$ be as in section \ref{ss:seifert}. If $\Gamma$ satisfies $b_0\geq d$ then a corresponding weighted homogeneous singularity $(X,0)$ supported on this topological type is minimal rational. Moreover, in this case $\calS_{\Gamma}=\mathbb{N}$. Otherwise, in the non-trivial cases the Frobenius number of $\calS_{\Gamma}$ is expressed by the following result.

\begin{theorem}{\cite{stronglyflat}}
 If $b_0<d$ then one has
 \begin{equation}\label{eq:frob}
 f_{\calS_{\Gamma}}=\gamma + \frac{1}{|e|}-\check{s},
 \end{equation}
 where $\check{s}$ is the $E_0$-coefficient of 
 the unique minimal element of the Lipman cone  $\mathcal{S}'_{[Z_K+E_0^*]}:=\{\ell'\in L' \ | \ (\ell',E_v)\leq 0 \ \mbox{for all} \ v\in \V \ \mbox{and} \ [l']=[Z_K+E_0^*]\}$, given by the generalized Laufer's algorithm, see \cite[3.1.2]{stronglyflat}.
\end{theorem}

If $\Gamma$ is numerically Gorenstein (ie. $Z_K\in L$) and $\mathfrak{o}=1$ then in (\ref{eq:frob}) the `algorithmic term' $\check{s}$ vanishes and the corresponding semigroup is symmetric as it is clarified by the next proposition.

\begin{proposition} \label{thm_ngor0_sym_and_frob}
    Let $\Gamma$ be a numerically Gorenstein graph which 
    satisfies $\oo=1$. Then 
     $\calS_{\Gamma}$ is symmetric.  
    Moreover, 
    the Frobenius number of $\calS_{\Gamma}$ simplifies to  
    \begin{equation}\label{eq:frobNGoro1}
    f_{\calS_{\Gamma}}=\al+\gamma.
    \end{equation}
\end{proposition}

\begin{proof}
The assumptions imply $1/|e|=\alpha$ and $\check{s}=0$, hence (\ref{eq:frob}) immediately gives the simplified form of $f_{\calS_{\Gamma}}$, cf. \cite[Corollary 3.2.12 or Example 6.2.4(1)]{stronglyflat}.
\vspace{0.1cm}

The proof of the symmetry goes analogously as in the case of strongly flat semigroups presented in \cite[4.1.1]{stronglyflat}. For the sake of completeness we will clarify the details here as well.

One needs to verify  that $\ell\in \calS_\Gamma$ if and only if $f_{\calS_\Gamma}-\ell\not\in\calS_\Gamma$ for
    every $\ell\in\Z.$ Using the quasi-linear function $N(\ell)$ and the expression (\ref{eq:frobNGoro1}) of the Frobenius number, this reads as 
\begin{equation}\label{eq:sym}
    N(\ell)\ge 0 \ \ \mbox{ if and only if } \ \ N(\al+\gamma-\ell)<0 
    \mbox{ for every } \ell\in\Z.     
\end{equation}
   
 Since $\G$ is numerically Gorenstein,  by Proposition \ref{prop:INEQ}(f) we have  $N(\ell)+N(\gamma-\ell)=-2$. On the other hand, part (d) of the same proposition deduces that $N(\al+\gamma-\ell)=N(\gamma-\ell)+\oo=N(\gamma-\ell)+1$,  hence  $N(\ell)+N(\al+\gamma-\ell)=-1$. Since $N(\ell)$ and $N(\al+\gamma-\ell)$ are integers one gets   (\ref{eq:sym}).
\end{proof}

\subsubsection{\bf Representatives of strongly flat semigroups}

Let  $M=\Sigma(\alpha_1,\dots,\alpha_d)$ be a Seifert integral homology sphere. Thus, $\alpha_1,\ldots,\alpha_d \geq 2$ ($d\geq 3$) are pairwise relatively prime integers and both $b_0$ and $(\omega_1,\ldots,\omega_d)$ are uniquely determined by the Diophantine equation $\alpha(b_0-\sum_{i=1}^d\omega_i/\alpha_i)=1$. If we consider the integers  $a_i:=\alpha/\alpha_i$
then the greatest common divisor of $a_1, \ldots, a_{i-1}, a_{i+1},\ldots, a_d$
 is $\alpha_i$, hence the system $\{a_i\}_{i=1}^d$ generates a strongly flat semigroup
 $G(a_1, \ldots, a_d)$. In fact, in this case $\mathcal{S}_M=G(a_1, \ldots, a_d)$.

\begin{thm}{\cite{stronglyflat}}\label{thm:Sgen}
The strongly flat semigroups with at least three generators are representable. They can be represented by the Seifert integral homology spheres.
\end{thm}
In particular, the theorem implies that strongly flat semigroups (with $d\geq 3$) are semigroups associated with numerically Gorenstein graphs with $\oo=1$. Consequently, they are symmetric and their Frobenius number is expressed by the formula $ f_{G(a_1,\ldots,a_n)}=\alpha+\gamma$. Its identification  with the bound  $B(a_1,\dots, a_d)$ can be seen using  (\ref{eq:B}), (\ref{eq:sei2}) and by noticing that in this case ${\rm lcm}(a_1,\ldots, a_d)=\alpha$.

Since the representative of a semigroup is not unique (see Remark \ref{rem:toprel}), we say that the Seifert integral homology sphere is the {\it canonical representative} of the corresponding strongly flat semigroup.

\section{Further observations on representable semigroups}\label{s:rpsg}

Let $\Gamma$ be a star-shaped minimal good dual resolution graph  of a normal weighted homogeneous singularity link is a rational homology sphere (cf. \ref{ss:seifert}). For simplicity, in the sequel we refer to such graphs as {\it SSR graphs}.

We consider the  representable semigroup $S_{\Gamma}$ defined by the quasi-linear function $N(\ell)$. In this section we make some observations which, on one hand, clarify the representability of the semigroups $G(p,q)$, on the other hand, give some bounds for representable semigroups which will be crucial in the forthcoming sections.

\subsection{The sum of SSR graphs}\label{sec:op}
%

Let $\G_1$ and $\G_2$ be two SSR graphs with Seifert invariants
\[
    Sf_1=\left(-b_0, (\al_i,\w_i )_{i=1}^m \right)
    \quad\text{and}\quad
    Sf_2=\left(-c_0, (\beta_j,w_j )_{j=1}^n \right).
\]
We define the sum $\G:=\G_1+\G_2$ of these two graphs as the SSR graph determined by the
Seifert invariants
\[
    Sf=\left(-b_0-c_0, (\al_i,\w_i )_{i=1}^m,(\beta_j,w_j )_{j=1}^n  \right).
\]
Note that if $e_1,e_2$ and $e$ are the orbifold Euler
numbers of $\G_1,\G_2$ and $\G$ respectively,
then $e=e_1+e_2<0$. Hence the sum is well defined.

Now, we study this sum from the perspective of the associated numerical semigroups. 
Denote the quasi-linear functions associated with the 
graphs $\G_{1},\G_2$ and $\G$ by $N_{1},N_2$ and $N$. Then the 
above construction yields $N(\ell)=N_1(\ell)+N_2(\ell)$. This provides an upper and a lower 
bound for  $\Se_{\G}$ 
in terms of $\Se_{\G_1}$ and $\Se_{\G_2}$, 
given by the next lemma. Recall that  for two sets $A,B\subset \mathbb{N}$ we denote
their sum by $A+B:=\{a+b:a\in A,\,b\in B\}$. In particular, if $A,B$ are numerical semigroups then $A+B$ is so.

\begin{lemma}
    \label{thm:sumbound}
    If $\G_1$ and $\G_2$ are SSR graphs of then
    \begin{equation}
        \label{eq:sumbound}
        \Se_{\G_1}\cap\Se_{\G_2}\subset 
        \Se_{\G_1+\G_2}
        \subset\Se_{\G_1}+\Se_{\G_2}.        
    \end{equation} 
\end{lemma}
\begin{proof}
    Since the quasi-linear function 
    associated with $\G_1+\G_2$ is $N_1+N_2$, one shows that 
    $$
    \Se_{\G_1+\G_2}
    =\{\ell\in\mathbb{N}:N_1(\ell)+N_2(\ell)\ge 0\}\supset \{\ell\in\mathbb{N}:N_1(\ell)\ge 0\mbox{ and }
    N_2(\ell)\ge 0\}=  \Se_{\G_1}\cap\Se_{\G_2}.
    $$
    Furthermore, we can write  
    $$\Se_{\G_1}+\Se_{\G_2}=\langle \calS_{\G_1}\cup \calS_{\G_2}\rangle=\langle \h{\ell\in\mathbb{N} :
        N_1(\ell) \ge 0 \text{ or } N_2 (\ell)\ge 0} \rangle,$$
     which clearly implies the inclusion $\Se_{\G_1+\G_2}\subset \Se_{\G_1}+\Se_{\G_2}$.   
   \end{proof}

\subsubsection{}\label{sss:multip} In particular, when $\Se_{\G_1}=\Se_{\G_2}$, we get identity
in (\ref{eq:sumbound}) implying  
$\Se_{2\G_1}:=\Se_{\G_1+\G_1}=\Se_{\G_1}$.
More generally, for an arbitrary  SSR graph $\G$ and 
$k\in \mathbb{N}^*$, we get $\Se_{k\G}=\Se_{\G}$, where 
\[
k\G:=\underset{k-\text{times}}{\underbrace{\G+\G+\dots+\G}}.    
\]

\begin{remark}\label{rem:toprel}
This also shows that the representaion of a numerical semigroup $\Se$
as the semigroup associated with an SSR graph is not unique. Moreover, one can construct an 
infinite family of representatives, which do not share immediate topological properties.
In particular, the associated semigroup we consider does not characterize completely the topological
type of the singularity. 
\end{remark}

%
%

\subsection{Example: representatives of $G(p,q)$} \label{s:gpq}

In this part we discuss the representability of the numerical semigroups $G(p,q)$ generated by two elements. These semigroups are strongly flat, but they are excluded from the result of \cite{stronglyflat}, cf. Theorem \ref{thm:Sgen}. 

The first remark towards to this direction is that in the definition of representable semigroups one can also allow SSR graphs with less than $3$ legs. This would also consider as possible representative links of cyclic quotient singularities (lens spaces), or even $S^3$, with a well-chosen Seifert structure. If we allow these cases, than there is a very natural choice for $G(p,q)$ which is as follows. 

Let $(C,0)\subset(\mathbb{C}^2,0)$ be an irreducible plane curve singularity defined by the analytic function germ $f:(\C^2,0)\to(\C,0)$. 
$(C,0)\subset(\mathbb{C}^2,0)$ admits a minimal good embedded resolution 
whose associated graph $\Gamma_f$ is a connected, negative definite tree
with an extra arrow representing the strict transform 
of $C$. The link of $(C,0)\subset(\mathbb{C}^2,0)$ is 
an algebraic knot $K\subset S^3$,
whose isotopy type can be completely characterized by many invariants such as: embedded resolution graph,
semigroup of $(C,0)$, Puiseaux pairs, Newton pairs, linking pairs or
the Alexander polynomial of the knot $K\subset S^3$. More details can be found in general references such as \cite{art_plane_sing,book_wall}.

Now assume that $(C,0)\subset(\C^2,0)$ has exactly one Puiseaux pair $(p,q)$, which means that the normal form of the defining equation is exactly 
$x^p+y^q=0$. In this case the graph $\Gamma_f$ has the following shape 
\begin{figure}[h!]
 \begin{center}
     \begin{tikzpicture}[scale=.6]
    \draw [line width=.5pt] (7.,0.)-- (9.,0.);
    \draw [line width=.8pt,dotted] (5.5,0.)-- (6.5,0.);
   \draw [->,line width=.5pt] (9,0) -- (10.5,0);
 \draw [line width=.5pt] (9.,0.)-- (9.,-1.3);
  \draw [line width=.5pt] (9,-2.2)-- (9,-3.5);
\draw [line width=.8pt,dotted] (9,-1.5)-- (9,-2);
\draw [line width=.5pt] (3.,0.)-- (5.,0.);
    \draw [fill=black] (3,0) circle (0.1);
    \draw[color=black] (2.9585299006542685,0.5) node {$-v_l$};
    \draw [fill=black] (4.5,0) circle (0.1);
    \draw[color=black] (4.49,0.5) node {$-v_{l-1}$};
    \draw [fill=black] (7.5,0) circle (0.1);
    \draw[color=black] (7.47,0.5) node {$-v_1$};
    \draw [fill=black] (9.,0.) circle (0.1);
    \draw[color=black] (8.8,0.5) node {$-1$};
    \draw [fill=black] (9,-3.5) circle (0.1);
    \draw[color=black] (10,-3.5) node {$-u_k$};
    \draw [fill=black] (9,-1) circle (0.1);
    \draw[color=black] (9.8,-1) node {$-u_1$};
    \draw [fill=black] (9,-2.5) circle (0.1);
    \draw[color=black] (10,-2.5) node {$-u_{k-1}$}; 
    \end{tikzpicture}
    \end{center}
\end{figure}
where the decorations can be determined from $(p,q)$. Indeed,  
if we introduce the numbers $0<\w_p<p$ and $0<\w_q<q$ uniquely
determined by  the Diophantine equation $pq-\w_p q-\w_q p=1$,  then the  negative continued 
fractions $p/\w_p=[u_1,\dots,u_k]$ and $q/\w_q:=[v_1,\dots,v_l]$ give
the corresponding decorations of $\Gamma_f$. 

Another important invariant of an irreducible plane curve singularity is its
numerical semigroup $\Se_f$. In this case this is $G(p,q)$. Moreover, by \cite{LSz_module} $\Se_f$ can be written as
\begin{equation}
    \label{eq:planesing_asN}
    \Se_f
=\{\ell\in\mathbb{N}: N(\ell)\ge 0\},
\end{equation}
where $N(\ell):=\ell-\lceil\w_p \ell/ p\rceil-\lceil\w_q \ell/q\rceil$ is the quasi-linear function associated with the Seifert structure $Sf=(-1, (p,\w_p),(q,\w_q))$ of $S^3$, specified by the negative definite plumbing graph $\widetilde{\Gamma}_f:=\Gamma_f\setminus\{\mbox{arrow}\}$. 

If we restrict ourselves to SSR graphs with at least 3 legs only, then we can use $\widetilde{\Gamma}_f$ and \ref{sss:multip} to find a suitable representative. (Note that in fact \ref{sss:multip} can also be extended to SSR graphs with less than $3$ legs). Namely, for an arbitrary $k\geq 2$, $G(p,q)$ is represented as the semigroup associated with the SSR graph $k\widetilde{\Gamma}_f$, defined by the Seifert invariants $Sf=(-k,k\times (p,\w_p), k\times (q,\w_q))$. (Here the notation means that $(p,\w_p)$, as well as $(q,\w_q)$ appears $k$ times.) Thus we have concluded the following.

\begin{cor}\label{thm:twogen}
    The numerical semigroup $G(p,q)$ is representable.     
\end{cor}

\begin{remark}
Now, by Theorem \ref{thm:Sgen} and Theorem \ref{thm:twogen}, the representability of strongly flat semigroups is clarified completely. 
\end{remark}

\subsection{Bounds for representable semigroups}\label{sec:bounds}

    Let $\Se$ be a numerical semigroup and $k\in\mathbb{N}^*$. We consider the numerical semigroup 
    $$\Se/k:=\{\ell\in\mathbb{N}:k\ell\in\Se\}$$
      which is called {\it the quotient} of 
    $\Se$ by  $k$.

Next we analyze $\Se/k$  from the perspective of the quasi-linear function. Let $\calS$ be a representable semigroup  and $N(\ell)$ its associated quasi-linear function and $k\in \mathbb{N}^*$. We set $N^{(k)}(l):=N(k l)$.
Then the semigroup associated with $N^{(k)}$, given by 
\[
\{\ell\in\mathbb{N}:N(k \ell)\ge 0\}=
\{\ell\in\mathbb{N}:k \ell \in \Se_N\},
\]
 is exactly $\Se/k$. Moreover, one can prove that the quotient is also representable.
 
\begin{lemma}\label{lem:quotientrep}
The semigroup  $\Se/k$ is representable for any representable semigroup $\Se$ and $k\in \mathbb{N}^*$.
\end{lemma}

\begin{proof}
 We represent $\Se$ by a normal weighted homogeneous surface singularity $(X,0)$ with minimal good dual resolution graph $\Gamma$ whose central vertex is denoted by $E_0$. Then the local algebra is expressed as $R_X=\oplus_{\ell\geq 0} H^0(E_0, \calO_{E_0}(D^{(\ell)}))$, see (\ref{eq:DPD}) and the definition of $D^{(\ell)}$ therein. Then the Veronese subring $R_X^{(k)}:=\oplus_{\ell\geq 0} H^0(E_0, \calO_{E_0}(D^{(k\ell)}))$ is also normal and therefore it corresponds to a normal weighted homogeneous surface singularity. Moreover, the associated semigroup is exactly $\Se/k$.
\end{proof}

\begin{remark}\label{rem:quotient}
 For the previous lemma one can give a combinatorial proof too which constructs explicitely from a representaive $\Gamma$ of $\Se$ a representative  $\Gamma^{(k)}$ of $\Se/k$.

We fix $k\geq 1$ and assume that $\Gamma$ has Seifert invariants $ Sf=(-b_0, (\al_i,\w_i )_{i=1}^n)$. Then the induced quasi-linear function of the quotient $\calS/k$  is expressed as 
 $$N^{(k)}(\ell)=N(k\ell)=b_0k\ell - \sum_{i=1}^n \Big\lceil \frac{k\ell \w_i}{\al_i}\Big\rceil.$$
 For every $i=1,\dots,n$ we consider  $0\leq r_i <\al_i$ satisfying $k\w_i \equiv r_i$ (mod $\al_i$). 
Then one gets $N^{(k)}(\ell)=(k|e|+\sum_{i=1}^n r_i/\al_i)\ell - \sum_{i=1}^n \lceil r_i/\al_i\rceil$. First of all, notice that $k|e|+\sum_{i=1}^n r_i/\al_i\in \Z_{>0}$ and  the new orbifold Euler number $e^{(k)}=ke$ is negative since $e<0$. Hence, we can associate with $N^{(k)}$ a negative definite star-shaped graph $\Gamma^{(k)}$ with Seifert invariants  $ Sf^{(k)}=(ke-\sum_{i=1}^n r_i/\al_i, (\al_i,r_i )_{i=1}^n)$ which represents the quotient semigroup $\calS/k$.
 
\end{remark}

\begin{example}\label{ex:G35_2}
As an illustration of the previous construction we consider the semigroup $G(3,5)$. We claim that it can be represented by the graph on the left hand side of Figure \ref{fig:1} (see section \ref{s:gpq}). By Remark \ref{rem:quotient} the associated quasi-linear function of the quotient semigroup $G(3,5)/2$ is written as $N^{(2)}(\ell)=2\ell-2\lceil\ell/5\rceil-2\lceil 2\ell/3\rceil$ which provides the graph drawn on the right hand side of Figure \ref{fig:1}. Moreover, since the set of gaps of this quotient is $\{1,2\}$, we have $G(3,5)/2=G(3,4,5)$.
\begin{figure}[h!]
\begin{center}
\begin{minipage}{5cm}
\begin{tikzpicture}[scale=.6]
\coordinate (v0) at (6,0);
\draw[fill] (v0) circle (0.1);
\coordinate (v11) at (5,-1);
\draw[fill] (v11) circle (0.1);
\coordinate (v12) at (4,-2);
\draw[fill] (v12) circle (0.1);
\coordinate (v21) at (4.5,0);
\draw[fill] (v21) circle (0.1);
\coordinate (v31) at (7.5,0);
\draw[fill] (v31) circle (0.1);
\coordinate (v41) at (7,-1);
\draw[fill] (v41) circle (0.1);
\coordinate (v42) at (8,-2);
\draw[fill] (v42) circle (0.1);
\draw  (v12) edge (v0);
\draw  (v21) edge (v0);
\draw  (v31) edge (v0);
\draw  (v42) edge (v0);
\draw[color=black] (6,0.4) node {\small $-2$};
\draw[color=black] (4.3,-1) node {\small $-2$};
\draw[color=black] (3.3,-2) node {\small $-3$};
\draw[color=black] (4.5,0.4) node {\small $-3$};
\draw[color=black] (7.5,0.4) node {\small $-3$};
\draw[color=black] (7.7,-1) node {\small $-2$};
\draw[color=black] (8.7,-2) node {\small $-3$};
\end{tikzpicture} 
\end{minipage}
\begin{minipage}{5cm}
\begin{tikzpicture}[scale=.6]
\coordinate (v0) at (6,0);
\draw[fill] (v0) circle (0.1);
\coordinate (v11) at (5,-1);
\draw[fill] (v11) circle (0.1);
\coordinate (v12) at (4,-2);
\draw[fill] (v12) circle (0.1);
\coordinate (v21) at (4.5,0);
\draw[fill] (v21) circle (0.1);
\coordinate (v31) at (7.5,0);
\draw[fill] (v31) circle (0.1);
\coordinate (v41) at (7,-1);
\draw[fill] (v41) circle (0.1);
\coordinate (v42) at (8,-2);
\draw[fill] (v42) circle (0.1);
\draw  (v12) edge (v0);
\draw  (v21) edge (v0);
\draw  (v31) edge (v0);
\draw  (v42) edge (v0);
\draw[color=black] (6,0.4) node {\small $-2$};
\draw[color=black] (4.3,-1) node {\small $-2$};
\draw[color=black] (3.3,-2) node {\small $-2$};
\draw[color=black] (4.5,0.4) node {\small $-5$};
\draw[color=black] (7.5,0.4) node {\small $-5$};
\draw[color=black] (7.7,-1) node {\small $-2$};
\draw[color=black] (8.7,-2) node {\small $-2$};
\end{tikzpicture} 
\end{minipage}
\end{center}
\caption{A representation of $G(3,5)$ and $G(3,4,5)$}
\label{fig:1}
\end{figure}
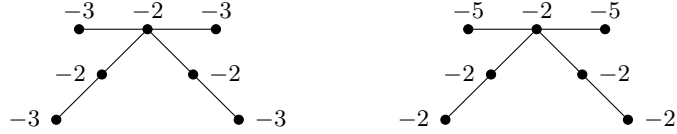
\end{example}

In the sequel we will give `bounds' for representable semigroups. It will be crucial in the next section for the characterization of flat semigroups.

   Let $\G$ be a SSR graph with Seifert invariants 
  \begin{equation}
    \label{eq:sf_multiple}
    Sf=\big(-b_0,s_1\times(\al_1,\w_1),\dots,s_n\times(\al_n,\w_n)
    \big), 
  \end{equation}
   where $(\al_i,\w_i)\neq(\al_j,\w_j)$  for different indices $i\neq j$. Here $s_i$ stands for the `multiplicity' of the $(\al_i,\w_i)$-type leg. Then, regarding its associated semigroup $\calS_{\Gamma}$ we obtain the following result.

   \begin{thm}\label{thm:numbound}
       One has the following inclusions
       \begin{equation}\label{eq:incl}
        G(\al,s_1 \al_{-1},\dots, s_n \al_{-n})\subset \Se_\G
        \subset G(\al,s_1 \al^*_1,\dots,s_n \al^*_n)/\oo,
        \end{equation}
        where $\al^*_i:=\al/\al_i$, $\al_{-i}=\mathrm{lcm}_{j\neq i}(\al_j)$ and the bounds are considered as submonoids of $\mathbb{N}$.
   \end{thm}
   \begin{proof}
       We will prove first the upper bound.  Using (\ref{eq:sei2}) and the definition of $N(\ell)$  one writes the expression 
       \begin{equation}
           \label{eq:l_as_N}
           \oo \ell=\alpha N(\ell)+\sum_{i=1}^n  s_i \al^*_i \cdot \al_i f\left(\frac{\w_i \ell}{\al_i}\right),
       \end{equation} 
       where $f(x)=\lceil x\rceil-x$ for any $x\in \mathbb{Q}$.
       Note that if $N(\ell)\geq 0$ then the right-hand side 
       of (\ref{eq:l_as_N}) is a nonnegative linear form of the generators  $\al$ and $\{s_i\al^*_i\}_{i=1}^n$. This implies that  
       \[
        \Se_{\G}\subset\h{\ell\in\mathbb{N}:\oo \ell\in G(\al,s_1 \al^*_1,\dots,s_n \al^*_n)}=
        G(\al,s_1 \al^*_1,\dots,s_n \al^*_n)/\oo, 
       \]
where $G(\al,s_1 \al^*_1,\dots,s_n \al^*_n)$ is  a submonoid of $\mathbb{N}$, not necessarily a numerical semigroup. 

In order to see the lower bound we can proceed as follows. For any fixed $i\in \{1,\dots,n\}$ the definition of the orbifold Euler number gives the expression
        \[
        s_i \frac{\w_i}{\al_i}=b_0+e-\sum_{j\neq i} s_j \frac{\w_j}{\al_j},
         \]
which is used to deduce the following inequality:
        \begin{align*}
            N(s_i\al_{-i})&=b_0 s_i\al_{-i}-
            \sum\limits_{j\neq i}s_j\left\lceil s_i \al_{-i}\frac{\w_j}{\al_j}\right\rceil -s_i\left\lceil\al_{-i} \left(s_i\frac{\omega_i}{\al_i}\right)\right\rceil\\       
            &=b_0 s_i \al_{-i}-
            s_i\sum\limits_{j\neq i} s_j \al_{-i} \frac{\w_j}{\al_j}
            -s_i\Big(b_0 \al_{-i}-\sum\limits_{j\neq i} \al_{-i} s_j\frac{\w_j}{\al_j}
            + \left\lceil e \al_{-i}\right\rceil\Big)\\
            &=-s_i\left\lceil e \al_{-i}\right\rceil\ge 0.
        \end{align*}
Furthermore, one also has the properties $N(\ell_1+\ell_2)\geq N(\ell_1)+N(\ell_2)$ and $N(\al)=\oo>0$ (Proposition \ref{prop:INEQ}(c)) which imply that $G(\al, s_1\al_{-1},\dots,s_n\al_{-n})\subset S_{\Gamma}$.       
   \end{proof}
   
\section{Representability of flat semigroups}{\label{sec:flat}}
We consider a special case in (\ref{eq:incl}) and characterize those  representable semigroups
which realize the bounds. It turns out that these are exactly the {\em flat semigroups}. In this section we will provide their characterization, prove that they are representable and discuss about their `canonical geometric representatives'.

\subsection{}
Consider the inclusions from (\ref{eq:incl}) and assume that $\oo=1$, $\al_{-i}=\al^*_i$  and $\gcd(s_i,\al_i)=1$ for every $i$. In this case the bounds are in fact numerical semigroups, they coincide and  (\ref{eq:incl}) becomes an identity. On the other hand, the condition $\al_{-i}=\al^*_i$  can be achieved exactly when the 
numbers $\al_i$ are pairwise relatively primes. Hence we get  $\al_{-i}=\al^*_i=\prod_{j\neq i}\al_j$ for which we will use the unified notation $\z{\al}_i$. 

In summary, one deduces 
the following consequence: 
if the graph $\Gamma$ defined by the Seifert invariants
$
     Sf=\big(-b_0,s_1\times(\al_1,\w_1),\dots,s_n\times(\al_n,\w_n)\big)    
$
satisfies $\oo=1$,  the numbers $\{\al_i\}_{i=1}^n$ are 
pairwise relatively prime integers and $\gcd(s_i,\al_i)=1$ for every $i$, 
then 
\begin{equation}
    \label{eq:sg_flat}
    \Se_\G=G(\al,s_1\z{\al}_1,\dots,s_n\z{\al}_n). 
\end{equation}
Moreover, in this case $\Se_\G$ is a flat semigroup.
Indeed, using  the notations from section \ref{ss:classsgp} one gets
$\q_0=\gcd(\{s_j\}_j)$ and $q_i=\al_i$ for any $i\geq 1$. This implies that  $\widehat{q}_0=\al_1\dots\al_n=\alpha$, $\widehat{q}_i=q_0\cdot\widehat{\al}_i$ for $i\geq 1$, $\widehat{s}_0=1$ and $\widehat{s}_i=s_i/q_0$ for $i\geq 1$, 
hence it clearly satisfies the flatness condition (cf. Definition \ref{df:flatnes}) at $i=0$.
\vspace{0.2cm}

Now, we start with a presentation $G(a_0,\dots,a_n)$ of a flat semigroup and  assume that $\widehat{s}_0=1$. Note that $\{a_0,\dots,a_n\}$ is not necessarily the minimal set of generators. In particular, when $n=1$, for the next construction we have to use a presentation with at least three generators, eg. $G(a_0a_1,a_0,a_1)$.

The chosen set of generators can be read as 
$\{\widehat{q}_0,\widehat{s}_1\widehat{q}_1,\dots, \widehat{s}_n\widehat{q}_n\}$ 
where we set the numbers $\q_i:=\gcd(a_0,\dots,a_{i-1},a_{i+1},\dots, a_n)$, $\widehat{\q}_i:=\prod_{j\neq i} \q_j$ and $\widehat{s}_i:=a_i/\widehat{\q}_i$ for all $i\in\{0,\dots,n\}$. Note that $\widehat{s}_i$ is an integer since $\gcd(\q_i,\q_j)=1$ for every $ i\neq j$, and one also follows that $\gcd(\widehat{s}_i,\q_i)=1$.  By setting $\al_i:=q_i$ for $i\geq 1$ and $s_i:=q_0\cdot \widehat{s}_i$ for every $i\geq 0$ one gets that $\{\al_i\}_i$ are pairwise relatively primes and $\gcd(s_i,\al_i)=1$ if $i\geq 1$. Moreover, one identifies $\widehat{q}_0=\alpha$ and $\widehat{s}_i\widehat{q}_i=s_i\widehat{\al}_i$ for every $i\geq 1$, hence the semigroup is presented  in the form of (\ref{eq:sg_flat}).  

The previous argument provides the `arithmetical' characterization of flat semigroups which was also proved in \cite{art_rch}.
\begin{thm}[\cite{art_rch}]
    $\Se$ is a flat semigroup if and only if there exist pairwise relatively prime integers $\al_i\geq 2$ ($i\in \{1,\dots,n\}$) such that $\Se$ can be presented as $G(\al,s_1\z{\al}_1,\dots,s_n\z{\al}_n)$ where $\alpha=\alpha_1\dots\al_n$, $\z{\al}_i=\prod_{j\neq i} \al_j$ and $\gcd(\al_i,s_i)=1$ for every $i$.
\end{thm} 
Next we show that, once the presentation (\ref{eq:sg_flat}) is fixed, there exists a canonical way to represent a flat semigroup as a semigroup associated with an SSR graph.
\begin{thm}
    \label{thm_flat_rep}
    Every flat semigroup is representable.
\end{thm}
\begin{proof}
Let $\calS=G(\al,s_1\z{\al}_1,\dots,s_n\z{\al}_n)$ be a fixed presentation of the flat semigroup $\calS$. We would like to find the appropiate $b_0\geq 1$ and $\w_1,\dots,\w_n$ such that $0<\omega_i<\al_i$ and  the Seifert invariants 
\begin{equation}\label{eq:Sfcr}
Sf=(-b_0,s_1\times(\al_1,\w_1),\dots,s_n\times(\al_n,\w_n))
\end{equation}     
define a negative definite star-shaped plumbing graph $\Gamma$ wiht $\oo=\al|e|=1$.  Note that this later condition is equivalent with $b_0-\sum_i s_i \w_i/ \al_i= 1/\al$.   

If we ignore the $s_i$-multiplicities, then the Diophantine equation $\al(\widetilde{b}_0-\sum_i \widetilde{w}_i/\al_i)=1$ has a unique solution $\widetilde{b}_0,\widetilde{\w}_1,\dots \widetilde{\w}_n$, see \cite{stronglyflat}.
 In addition, if $\widetilde{\w}_i$ is divisible by $s_i$ for every $i$,
 then one writes $\widetilde{b}_0-\sum_i s_i \frac{\widetilde{\w}_i/s_i}{\al_i}=1/\al$, hence by setting $\w_i:=\widetilde{\w}_i/s_i$  the construction
 is finished. However, in general, this divisibility does not hold and 
 we have to perturb the initial solution $\widetilde{b}_0,\widetilde{\w}_1,\dots,\widetilde{\w}_n$ as follows.
 
 Note that for arbitrary $k_1,\dots,k_n\in \mathbb{Z}_{\geq 0}$ we can write the 
 following identity
 \[
    \frac{1}{\al}=\widetilde{b}_0+k_1+k_2+\dots+k_n-\sum_{i=1}^n \frac{k_i \al_i+\widetilde{\w}_i}{\al_i}.
 \]
Since the positive integers $s_i$ and $\al_i$ are relatively prime,
 we choose $k_i$ to be the unique  
non-negative solution of  the equation $k_i \al_i+\widetilde{\w}_i\equiv 0\, (\mathrm{mod } \ s_i)$ such that $0\leq k_i<s_i$.
 In this case $k_i \al_i+\widetilde{\w}_i$ is divisible by $s_i$, hence we can define  
 \[
    \w_i:=(k_i \al_i+\widetilde{\w}_i)/s_i\in\mathbb{N}\ \ \mbox{for every } i\quad
    \mbox{ and}\quad 
    b_0:=\widetilde{b}_0+k_1+\dots+k_n >0. 
 \] This yields that
 \[
 0<\w_i<\frac{1}{s_i}\left((s_i-1)\al_i+\al_i\right)=\al_i\mbox{ for every } 1\le i\le n \quad\text{ and }\quad     e=b_0-\sum_{i=1}^n s_i \frac{\w_i}{\al_i}=\frac{1}{\al},
 \] hence we get $\oo=1$ which finishes the proof.

\end{proof}

\begin{remark}\label{rem:CRSF}
In the case when $s_i=1$ for any $i$, the graph constructed in the proof of Theorem \ref{thm_flat_rep} is the canonical representative of a strongly flat semigroup. This motivates the following definition.  
\end{remark}

 \begin{definition}
 We say that the representative constructed in Theorem \ref{thm_flat_rep} is a 
 {\it canonical representative} of the flat semigroup $\calS$ associated with its presentation $G(\al,s_1\z{\al}_1,\dots,s_n\z{\al}_n)$. 
 \end{definition}
 
 \begin{remark}\label{rem:canrep}
(a) \ The next example illustrates how does a canonical representative depend on the presentation $G(\al,s_1\z{\al}_1,\dots,s_n\z{\al}_n)$ of the flat semigroup. Once the presentation is fixed, it is unique by the previous proof. 
  
 Consider the numerical semigroup $\calS$ generated by  $a_0=6$, $a_1=15$ and $a_2=20$. One can check that $\calS$ is flat at the first two generators, ie. $\widehat{s}_0=\widehat{s}_1=1$. Hence, if we present it  first  as  $G(2\cdot 3,5\cdot 3,10\cdot 2)$ then  this provides a canonical representative with Seifert invariants $Sf=(-6,5\times (2,1),10\times (3,1))$. On the other hand, if we present it as $G(5\cdot 3, 2\cdot 3, 4\cdot 5)$, one gets a graph with Seifert invariants $Sf=(-3,4\times(3,1),2\times (5,4))$. 
One can also choose a non-minimal presentation such as $G(2\cdot 3\cdot 5, 1\cdot 15,2\cdot 10,1\cdot 6)$. In this case the associated canonical representative is defined by the Seifert invariants $Sf=(-2, (2,1),2\times (3,1),(5,4))$. 
 
(b) \ If we run the procedure from Theorem \ref{thm_flat_rep} for the strongly flat semigroup $G(p,q)$ presented as $G(pq,p,q)$, then the resulted canonical representative is be defined by the Seifert invariants  $Sf=(-b_0, (p,\omega_1), (q, \omega_2))$ where $b_0,\omega_1,\omega_2$ satisfy $pqb_0-q\omega_1-p\omega_2=1$. The last identity immediately implies that $b_0=1$ and the corresponding graph is $\widetilde{\Gamma}_f$, see section \ref{s:gpq}. In other words, the canonical representative is $S^3$ with the corresponding Seifert structure.

 \end{remark}

\subsection{} In the sequel, we discuss some  properties of 
the canonical representatives of flat semigroups. 
 \begin{proposition}\label{lm_flat_ngor}    
    A canonical representative of a flat
    semigroup is numerically Gorenstein.
 \end{proposition}
 \begin{proof}
    Let $\G$ be the canonical representative associated with a fixed presentation of 
    a flat semigroup and consider its Seifert invariants as in (\ref{eq:Sfcr}).  By Lemma \ref{lem:ngor} one has to check 
    that $\gamma\in \mathbb{Z}$ and $\gamma\equiv\omega_i'\ ({\rm mod}\ \alpha_i)$  for every $1\le i\le n$. 
    
  For the first, we recall that for a canonical representative we have $\al=\al_1\al_2\dots\alpha_n$ and $\oo=1$, hence $|e|=1/\al$. Then, from (\ref{eq:gamma}) we get 
     \begin{align*}
         \gamma&=\frac{1}{|e|}\cdot
         \Big(d-2-\sum_{\mbox{\tiny all} \ \al_j} \frac{1}{\al_j}\Big)
         =\al\cdot  
         \Big(d -2-\sum_{i=1}^{n}s_i \frac{1}{\al_i}\Big)=(d-2)\al-\sum_{i=1}^{n}s_i \widehat{\al}_i\in\mathbb{Z}.            
     \end{align*}
   
     Note that $\gamma\equiv\omega_i'\ ({\rm mod}\ \alpha_i)$ is equivalent with $ \w_i \gamma\equiv 1\,(\mathrm{mod }\,\al_i)$, see section \ref{ss:seifert}. The previous calculation gives the expression  
     \[
     \w_i \gamma= \al_i
 \Big((d-2)\widehat{\al}_i-\sum_{\substack{k=1\\k\neq i}}^n s_k\prod_{\substack{j=1\\j\notin \{i,k\}}}^n\al_j \Big)\w_i
 -s_i\widehat{\al}_i \w_i,   
     \]
     which tells us that  $\w_i \gamma \equiv -s_i\widehat{\al}_i\w_i\, (\mathrm{mod }\,{\al_i})$.
    On the other hand, the identity  $\oo=1$ reads as
    \[
    1=\al|e|= \al\Big(b_0-\sum_{k=1}^n s_k \frac{\w_k}{\al_k}\Big)= \al_i
    \Big(b_0\widehat{\al}_i-\sum_{\substack{k=1\\k\neq i}}^n s_k\w_k\prod_{\substack{j=1\\j\notin \{i,k\}}}^n\al_j \Big)
    -s_i\widehat{\al}_i\w_i,  
    \]
    which implies 
   $-s_i\widehat{\al}_i\w_i\equiv 1(\mathrm{mod }\,{\al_i})$. 
   Hence, $\w_i\gamma\equiv 1(\mathrm{mod }\,{\al_i})$ for any $i\in\{1,\dots,n\}$.
 \end{proof}

We can apply Proposition \ref{thm_ngor0_sym_and_frob} to deduce that a flat semigroup is symmetric. Furthermore, in this case the Frobenius number simplifies to $f_{\calS}=\alpha+\gamma$. This, expressed by the minimal set of generators reproves the formula (\ref{eq:Frobflat}) from \cite[Theorem 2.5]{art_rch}.

\begin{thm}\label{thm_flat_frob}
     If $\Se$ is a flat semigroup, minimally generated by $a_0,a_1,\dots,a_{n}$ ($n\geq 2$), then
    \begin{equation}
        \label{eq:frob_flat}
        f_{\calS}=\sum_{i=0}^{n}(q_i-1)a_i-\prod_{i=0}^{n} q_i,    
    \end{equation}
    where $q_i=\al_i=\gcd(a_0,\dots,a_{i-1},a_{i+1}\dots,a_{n})$.
 \end{thm} 
 \begin{proof}
 Using the previous discussions and the notation $\alpha:=\prod_{i=1}^n \al_i$, after a possible permutation of the generators, we can assume that
    $(a_0,a_1,\dots,a_{n})=(\al,s_1\z{\al}_1,\dots,s_n\z{\al}_n)$. Then one has
    \begin{align*}
     \gamma+\al&=\frac{1}{|e|}(d-2-\sum_{i=1}^n s_i\frac{1}{\al_i})+\al=\al(d-2-\sum_{i=1}^n s_i\frac{1}{\al_i}+1)=d\al-\sum_{i=1}^n s_i \widehat{\al}_i-\al\\ 
     &=\sum_{i=1}^n s_i \al_i\widehat{\al}_i-\sum_{i=1}^n s_i \widehat{\al}_i-\al=\sum_{i=1}^{n}(\al_i-1)s_i\widehat{\al}_i +(\al_{0}-1)\al-\al_{0}\al\\
     &=\sum_{i=0}^{n}(\al_i-1)a_i-\prod_{i=0}^{n} \al_i
    \end{align*}    
 \end{proof}

\subsection{The geometric canonical representatives}\label{s:ICIS}
In this section we construct explicit equations for weighted homogeneous surface singularities whose link (or minimal good dual resolution graph) is a canonical representative of a flat semigroup.
 
\subsubsection{\bf The universal abelian cover and the action of $H$}
\begin{lemma}\label{lem:H}
If $\Gamma$ is the the canonical representative of a flat semigroup $G(\al,s_1\z{\al}_1,\dots,s_n\z{\al}_n)$,  then  one has
$$H\simeq \oplus_{i=1}^n \Z_{\al_i}^{s_i-1}.$$
\end{lemma}

\begin{proof}
Recall that the canonical representative $\Gamma$ associated with the given  presentation  is defined by the Seifert invariants $Sf=(-b_0, s_1\times (\al_i,\w_i),\dots,s_n\times (\al_n,\w_n))$ where $b_0$ and $\w_i$ are contructed by Theorem \ref{thm_flat_rep}.

Let $E_0$ be the base element associated with the central node of $\Gamma$, and for simplicity, we will denote by $E_{j(i)}$ ($i\in\{1,\dots,n\}$, $j\in\{1,\dots,s_i\}$) the base elements associated with the end-vertices. The classes in $H= L'/L$ of the corresponding dual base elements will be denoted by $g_0:=[E^*_0]$ and $g_{j(i)}:=[E^*_{j(i)}]$. Then the group $H$ can be presented as 
$$H=\big\langle \, g_0,\{g_{j(i)}\}_{i,j}\,|\, b_0\cdot g_0=\sum_{i=1}^n\sum_{j=1}^{s_i} \omega_i\cdot g_{j(i)};
\, g_0=\alpha_i\cdot g_{j(i)}\ (1\leq i\leq n, \ 1\leq j\leq s_i)\big\rangle,$$
cf. Neumann \cite{neumann.abel}. 

Since $\oo=1$, one gets $g_0=0$ and the relations simplify to  $\sum_{i=1}^n\sum_{j=1}^{s_i} \omega_i\cdot g_{j(i)}=0$ and $\alpha_i\cdot g_{j(i)}=0$. From the former relation we deduce that $\z{\al_i}\w_i\cdot \sum_{j=1}^{s_i}g_{j(i)}=0$, while the later gives $\al_i\cdot \sum_{j=1}^{s_i}g_{j(i)}=0$. These imply that the order of $\sum_{j=1}^{s_i}g_{j(i)}$ divides both $\z{\al_i}\w_i$ and $\al_i$. Since $\{\al_i\}_i$ are pairwise relatively prime, this is possible if and only if $\sum_{j=1}^{s_i}g_{j(i)}=0$ for any $i\in\{1,\dots, n\}$. Therefore, we get 
\begin{equation}\label{eq:Hrel}
H\simeq \oplus_{i=1}^n\langle \, g_{1(i)},\dots,g_{s_i(i)}\,|\, \alpha_i\cdot g_{j(i)}=0;
\, \sum_{j(i)=1}^{s_i}g_{j(i)}=0\ (1\leq j\leq s_i)\rangle\simeq \oplus_{i=1}^n \Z_{\al_i}^{s_i-1}.
\end{equation}
\end{proof}

Now, let $(X,0)$ be a weighted homogeneous surface singularity whose minimal good dual resolution graph is $\Gamma$. Then there exists the universal abelian cover $(X^{ab},0)\to (X,0)$ that induces an unramified Galois covering $X^{ab}\setminus\{0\} \to X\setminus\{0\}$ with Galois group $H\simeq H_1(M,\Z)$ ($M$ is the link of $(X,0)$), ie. $(X,0)=(X^{ab}/H,0)$. Furthermore, by a theorem of Neumann \cite{neumann.abel} (see also \cite[5.1.40]{nembook}) this universal abelian cover $(X^{ab},0)$ is a Brieskorn complete intersection singularity, which can be given by the equations   
\begin{equation}\label{eq:Xab}
\{z=(z_{j(i)})\in \C^d \ | \ f_k:=\sum_{i=1}^n \sum_{j(i)=1}^{s_i} c^k_{j(i)} z_{j(i)}^{\al_i}=0, \ \ k=1,\dots, d-2\},
\end{equation}
where the $(d-2)\times d$ matrix $(c^k_{j(i)})$ has full rank. Here $d:=s_1+\dots +s_n$ and the variable $z_{j(i)}$ is assigned to the end-vertex $E_{j(i)}$ for any $i=1,\dots,n$, $j=1,\dots, s_i$. In the followings, we define the $H$-action on $X^{ab}$.

Consider the Pontrjagin dual $\hat{H}:=\Hom(H,S^1)$ of $H$ and let $\theta:H\to \hat{H}, [l']\mapsto e^{2\pi i (l',\cdot)}$ be the isomorphism between $H$ and $\hat{H}$. The $H$ acts on $\C^d$ by the diagonal action $\mathrm{diag}(\chi_{j(i)})_{i,j}:H\to \mathrm{Diag}(d)\subset GL_d(\C)$, where $\chi_{j(i)}:=e^{2\pi i (E^*_{j(i)},\cdot)}\in \hat{H}$ is the character corresponding to $[E^*_{j(i)}]$. Since the equations $f_k$ are eigenvectors we obtain an induced action on $X^{ab}$ too. By \cite[Theorem 2.1]{neumann.abel} this action is free off the origin and the orbit space $(X^{ab}/H,0)\simeq (X,0)$ (for the right choice of $c^k_{j(i)}$). 

\begin{remark}
Note that, although the complex structure of $(X^{ab},0)$ depends on the choice of the matrix $(c^k_{j(i)})$, its link $M^{ab}$ (as a Seifert $3$-manifold) is independent. Furthermore, by the result of \cite[Thm. 7.2]{JNSeifert} (see also \cite[5.1.17]{nembook}) one can show that  in our case the resolution graph $\Gamma^{ab}$ of $(X^{ab},0)$ inherits from $\Gamma$ the following structure (in the sequel all invariants of $M^{ab}$ will be marked by $\star^{ab}$):
\begin{itemize}
 \item If $s_i=1$ then the leg in $\Gamma$ with $(\al_i,\w_i)$ induces a leg in $\Gamma^{ab}$ with $\al^{ab}_i=\al_i$ and this will appear with multiplicity $s^{ab}_i=\prod_{l=1}^n\al_{l}^{s_l-1}$.
 \item If $s_i>1$ then for any $j(i)\in\{1,\dots,s_i$\} one gets $\al^{ab}_{j(i)}=1$ and $s^{ab}_{j(i)}=\al_i^{s_i-2}\prod_{l\neq i}\al_l^{s_l-1}$. Note that these legs completely dissapear since $\al^{ab}_{j(i)}=1$, however their multiplicity contribute to the genus of the central fiber in $M^{ab}$. 
 \item More precisely, one gets $g^{ab}=1+\frac{1}{2}\prod_{l=1}^n\al_l^{s_l-1}(\sum_{i,s_i>1}s_i(1-\frac{1}{\al_i})-2)$. Furthermore, $b_0^{ab}$ and $\w^{ab}_i$ can be determined by the formulae: $\w^{ab}_i\z{\al}_i\equiv -1 \,(\mathrm{mod }\, \al_i)$ and $-e^{ab}=\prod_{l=1}^n \al_l^{s_l-2}$. 
\end{itemize}
We emphasize that $M^{ab}$ is a rational homology sphere if and only if $s_i=1$ for any $i$. Or, equivalently, $\Gamma$ is the canonical representative of a strongly flat semigroup by Remark \ref{rem:CRSF}. 
\end{remark}

\subsubsection{\bf The equations of $(X,0)$}

Now we look at the induced action on the polynomial ring $R=\C[(z_{j(i)})_{i,j}]$ and the invariant subring $R^H$ of $R$. By considering the generators of the invariant monomials and their relations in $R^H$, in the sequel we will study  the possible equations for the analytic types of $(X,0)$.

Since the characters $\chi_{j(i)}$ generate $\hat{H}$, they satisfy the relations from (\ref{eq:Hrel}), namely one has 
\begin{equation}\label{eq:chrel}
\chi_{j(i)}^{\al_i}=1 \ \ \mbox{and} \ \ \prod_{j(i)=1}^{s_i}\chi_{j(i)}=1 \ \ \mbox{for any} \ i=1,\dots,n.
\end{equation}
A monomial $z^{a}:=\prod_{i}\prod_{j(i)} z_{j(i)}^{a_{j(i)}}$ is in $R^H$ if and only if $\prod_{j(i)}\chi_{j(i)}^{a_{j(i)}}=1$ for any $i$. By (\ref{eq:chrel}) one deduces that the monomials $z^{\al_i}_{j(i)}$ and $\prod_{j(i)=1}^{s_i}z_{j(i)}$ are in $R^H$ for any $i$ and $j(i)$. If there are any other generators, divided them with the already listed invariant monomials, they must have the form of $\prod_{j(i)=1}^{s_i}z^{a_{j(i)}}_{j(i)}$ for some $i$ with the exponents $(a_{j(i)})=(a_1,\dots,a_{s_i-1},0)$ where $0\leq a_{j(i)}<\al_i$. We claim that in this case $a_{j(i)}=0$ for $j(i)=1,\dots,s_i-1$ as well. 
Indeed, one has the identity $\prod_{j(i)=1}^{s_i-1}\chi^{a_{j(i)}}_{j(i)}=1$ which implies that $\sum_{j(i)=1}^{s_i-1}a_{j(i)}E^*_{j(i)}\in L$. In particular, $\sum_{j(i)=1}^{s_i-1}a_{j(i)}g_{j(i)}=0 \in H$. Then by the isomorphism from Lemma  \ref{lem:H} and the assumptions on $a_{j(i)}$ follows that $a_{j(i)}=0$. 

In summary, the generators associated with $i$ are as follows: if $s_i=1$ then $z_i:=z_{j(i)}$ is a generator; in the case $s_i>0$ we get $w_{j(i)}:=z^{\al_i}_{j(i)}$ for $j(i)=1,\dots,s_i$ and $w_i:=\prod_{j(i)=1}^{s_i} z_{j(i)}$. Then $R^H$ can be presented as $\C[z_i,w_i,w_{j(i)}]/I$ where the ideal $I$ is given by the relations
\begin{equation}\label{eq:X}
\begin{cases}
\sum_{i,s_i=1} c^k_{i} z_{i}^{\al_i}+\sum_{i,s_i\neq 1}\sum_{j(i)}c^k_{j(i)} w_{j(i)}=0, \ \ \ & k=1,\dots,d-2;\\
w_i^{\al_i}=\prod_{j(i)=1}^{s_i} w_{j(i)}  \ & \mbox{for every} \ i \ \mbox{with} \ s_i>1,
\end{cases}
\end{equation}
providing us the equations for the possible analytic types of $(X,0)$. 
Note that the first type of relations is coming from the equations (\ref{eq:Xab}) of $X^{ab}$, the second type of equations is given by the relations of the monoid algebra $\C[M_i]$ associated with the affine monoid $M_i=\langle (t_1:=\al_i,0,\dots,0), t_2:=(0,\al_i,0,\dots,0),\dots, (0,\dots,0,\al_i),t_{s_i+1}:=(1,1,\dots,1)\rangle\subset \Z^{s_i}_{\geq 0}$ for every $i$ with $s_i>1$. In other words, for a fixed $i$ the corresponding relations generates the toric ideal $I_i$ where $\C[M_i]\simeq \C[w_{j(i)},w_i]/I_i$. In fact, in our case one has $I_i=(w_i^{\al_i}-\prod_{j(i)=1}^{s_i} w_{j(i)})$. 
Indeed, the generators of the relations in terms of the monoid generators have the form of $\sum_{l\in I}a_l t_l=\sum_{m\in J}b_m t_m+b_{s_i+1}t_{s_i+1}$ for some $a_l,b_m,b_{s_i+1}\geq 0$ where $I,J\subset \{1,\dots,s_i\}$ and  $I\cap J=\emptyset$. Looking at the coordinates, this implies that  $b_m=0$, $I=\{1,\dots,s_i\}$ and $a_l\al_i=b_{s_i+1}$, which gives us the only generator $\sum_{l=1}^{s_i}t_l=\al_it_{s_{i}+1}$.

Note that many of the variables $w_{j(i)}$ can be eliminated and the number of equations in (\ref{eq:X}) can be reduced. In the followings, we will distinguish three cases depending on the number $K:=\#\{i \ : \ s_i=1\}$ of legs with multiplicity $1$.  
\vspace{0.3cm}

{\bf I.} \ $K=0$ \ \  In this case we have the linear system $\{\sum_{i,s_i\neq 1}\sum_{j(i)}c^k_{j(i)} w_{j(i)}=0\}_{k=1,\dots,d-2}$. Associated with two fixed, not necesarilly different indices $i_1$ and $i_2$ we choose $j_0(i_1)\in\{1,\dots,s_{i_1}\}$ and $j_0(i_2)\in \{1,\dots s_{i_2}\}$. For simplicity, we set the notations $x:=w_{j_0(i_1)}$ and $y:=w_{j_0(i_2)}$. Then, since the coefficients $\{c^k_{j(i)}\}$ are generic (ie. the matrix of the system has rank $d-2$), the other variables can be expressed linearly in terms of $x$ and $y$. Therefore, we get that $(X,0)\subset (\C^{n+2},0)$ is defined by the following equations
\begin{equation}
w_i^{\al_i}=\prod_{j(i)=1}^{s_i} (a_{j(i)}x + b_{j(i)} y )  \  \ \ \ i=1,\dots n,
\end{equation}
where $a_{j_0(i_1)}=b_{j_0(i_2)}=1$, $a_{j_0(i_2)}=b_{j_0(i_1)}=0$ and the other $a_{j(i)}, b_{j(i)}\in \C$ are generic coefficients.

{\bf II.} \ $K=1$ \ \ We simply write $z$ for the variable associated with the only $i$ with $s_i=1$, set also $\al:=\al_i$. On the other hand, we fix $i_0\in\{i \ : s_i>1\}$ and one of its associated variables will be denoted by $x:=w_{j_0(i_0)}$. The other variables are linearly expressed with $z^\al$ and $x$, hence the equations of $(X,0)\subset \C^{n+1}$ are
\begin{equation}
w_i^{\al_i}=\prod_{j(i)=1}^{s_i} (a_{j(i)}z^{\al} + b_{j(i)} x )  \  \ \ \ i\in\{i \ : \ s_i>1\},
\end{equation}
where $a_{j_0(i_0)}=0$, $b_{j_0(i_0)}=1$ and the others are generic.

{\bf III.} \ $K>1$ \ \ In the last case we choose $i_1,i_2$ with $s_{i_1}=s_{i_2}=1$ and denote their associated variables by $x:=z_{i_1}$ and $y:=z_{i_2}$. Then, one can express  $z_i^{\al_i}$ for $i\in \{i \ : \ s_i=1\}$ and $i\neq i_1,i_2$, as well as the variables $w_{j(i)}$ for every $i\in \{i \ : \ s_i>1\}$ and $j(i)$, linearly in terms of $x^{\al_{i_1}}$ and $y^{\al_{i_2}}$. Therefore we get that $(X,0)\subset (\C^n,0)$ is defined by the following equations
\begin{equation}
\begin{cases}
z_i^{\al_i}=p_ix^{\al_{i_1}}+q_i y^{\al_{i_2}} \ \ \ \ & i\in\{i \ : \ s_i=1\}\setminus \{i_1,i_2\};\\
w_i^{\al_i}=\prod_{j(i)=1}^{s_i} (a_{j(i)}x^{\al_{i_1}} + b_{j(i)} y^{\al_{i_2}} )  \  \ \ \ & i\in\{i \ : \ s_i>1\},
\end{cases}
\end{equation}
where $p_i,q_i, a_{j(i)}, b_{j(i)}$ are generic coefficients.

\begin{remark}
We emphasize that in all three cases the representatives $(X,0)$ are complete intersections. In particular, they are Gorenstein and Proposition \ref{lm_flat_ngor} follows automatically.  
\end{remark}

\begin{example} 
Consider the flat semigroup $G(6,15,20)$ discussed in Remark \ref{rem:canrep}. We look at its (last) canonical representative, which is defined by the Seifert invariants $Sf=(-2,(2,1), 2\times (3,1), (5,4))$. Then, by case III. of the above construction, we get a family of suspension hypersurface singularities defined by $(X_{a_i,b_i}=\{f(x,y,z)=(a_1x^2+b_1y^5)(a_2x^2+b_2y^5)+z^3=0\},0)\subset (\mathbb{C}^3,0)$ where $a_i,b_i\in \C$ are generic coefficients. In particular, if we consider the hypersurface singularity defined by the equation $x^4+y^{10}+z^3=0$ for example, one can check that its associated semigroup is $G(6,15,20)$. The other canonical representatives considered in Remark \ref{rem:canrep} provide (eg. by case I.) other families of complete intersections with the same associated semigroup.    
\end{example}

\section{The characterization of representable semigroups} \label{s:ch}

\subsection{Perturbation of the Seifert invariants and the characterization}

In this section we prove that representable semigroups are exactly the  quotients of  flat semigroups. 

In order to prove this result, two technical steps are needed: first we show that every Seifert invariant $(\al,\w)$ can be 
\textquoteleft perturbed\textquoteright\ without affecting the quasi-linear function $N(\ell)$; the  second step claims that the Seifert invariants $(-b_0,s_1\times(\al_1,\w_1),\dots,s_{n}\times(\al_{n},\w_{n}))$ can be changed to  $(-b_0,s_1\times(\al'_1,\w'_1),\dots,s_{n}\times(\al'_{n},\w'_{n}))$ in such a way that the associated semigroup remain stable, while $\al'_1,\dots\al'_n$ will be pairwise relatively prime and $\gcd(\al'_i,s_{i})=1$.

We start our discussion with the characterization of the later case.

\begin{thm}
     \label{flat:divisor}
     Let $\G$ be an SSR  graph defined by the Seifert invariants   
     \[
         Sf=\big(-b_0,s_1\times(\al_1,\w_1),\dots,s_n\times(\al_n,\w_n)
         \big) .    
     \]
     If the numbers $\al_i\geq 2$  are pairwise relatively prime integers and $\gcd(\al_i,s_i)=1$ for every  $i$,
     then its associated semigroup is a quotient of a flat semigroup.
     In fact,
     $\Se_{\G}=G(\al,s_1 \widehat{\al}_1,\dots,s_n\widehat{\al}_n)/\oo$.
 \end{thm}
 \begin{proof}
 First of all we observe that $\gcd(\al_i,\oo)\neq 1$ implies $\gcd(\al_i,s_i)\neq 1$, hence by the assumptions we must have $\gcd(\al_i,\oo)=1$ for every $i\in\{1,\dots,n\}$. Indeed, for a fixed $i$ the expression $\oo=\al b_0-\sum_{j\neq i}s_j\w_j\widehat{\al}_j - s_i\w_i\widehat{\al}_i$ implies the identity $s_i\w_i\widehat{\al}_i\equiv 0\,(\mathrm{mod }\,{\gcd(\al_i,\oo)})$, which simplifies to $s_i\w_i\equiv 0\,(\mathrm{mod }\,{\gcd(\al_i,\oo)})$, since $\gcd(\al_i,\al_j)=1$ for $j\neq i$. Then, by multiplying the last congruence with $\w_i'$ and using (\ref{eq:w'}) yields that $s_i\equiv 0\,(\mathrm{mod }\,{\gcd(\al_i,\oo)})$, which supports our claim.

    Now, using similar ideas as in the proof of the Theorem \ref{thm_flat_rep} we consider the non-negative integers $k_1,\dots,k_n$ as the solutions of the equations
     \begin{equation}\label{eq:cong1}
      k_1 \al_1+\w_1\equiv
     \dots\equiv k_n \al_n+\w_n\equiv 0\,(\mathrm{mod }\,{\oo}),    
     \end{equation}
     with $0\leq k_i< \oo$ for every $1\le i\le n$. 
     These, on one hand, allow us to define the numbers 
     \[
        \tilde{\w}_i=\frac{1}{\oo} (k_i\al_i+\w_i)\quad
        \mbox{ for every }  1\le i\le n,
     \]
     which satisfy $\gcd(\tilde{\w}_i,\al_i)=1$ and $0<\tilde{\w}_i<\al_i$. 
     On the other hand, they imply that $b_0+\sum_{i=1}^ns_ik_i \equiv 0\,(\mathrm{mod }\,{\oo})$ too, hence one can define the new central  decoration as $\tilde{b}_0=(b_0+\sum_{i=1}^ns_ik_i)/\oo$. Indeed, from the equations (\ref{eq:cong1}), one gets $\sum_{i=1}^n\al s_i k_i +s_i\w_i\widehat{\al}_i\equiv 0\,(\mathrm{mod }\,{\oo})$. Furthermore,  the definition of $\oo$ reads as $\oo=\al b_0-\sum_i s_i\w_i\widehat{\al}_i$ which,  applied to the previous equation,  gives $\al(b_0+\sum_i s_i k_i) \equiv 0\,(\mathrm{mod }\,{\oo})$.  Since $\gcd(\al,\oo)=1$ we deduce that  $(b_0+\sum_i s_i k_i) \equiv 0\,(\mathrm{mod }\,{\oo})$.
     
     We consider the SSR graph $\tilde{\G}$ defined by the Seifert invariants
     \[
       Sf=\big(-\tilde{b}_0,s_1\times(\al_1,\tilde{\w}_1),\dots,s_n\times(\al_n,\tilde{\w}_n)
    \big).     
    \]
    In the followings, any of the numerical data associated with $\tilde{\G}$
    will be distinguished by the \textquoteleft tilde\textquoteright\, notation, 
    eg.  $\tilde{e}$ will stand for the orbifold Euler number of $\tilde{\G}$. The first observation is that $\tilde{e}=e/o<0$, hence $\tilde{\G}$ is negative definite. Furthermore, $\tilde{\al}=\al$ and one implies  that  $\tilde{\oo}=1$. Then, by the assumptions of the theorem we conclude that $\tilde{\G}$ is the canonical representative of the flat semigroup $\calS_{\tilde{\Gamma}}=G(\al,s_1 \widehat{\al}_1,\dots, s_n\widehat{\al}_n)$. 
    
    On the other hand, if we look at the quasi-linear function  $\tilde{N}(\ell)=\tilde{b}_0 \ell -\sum_{i=1}^n s_i\lceil\tilde{\w}_i \ell/\al_i\rceil$ associated with $\calS_{\tilde{\Gamma}}$, one can see that 
    $ \tilde{N}^{(\oo)}(\ell)=\tilde{N}(\oo\ell)=N(\ell)$. Hence $\Se_{\G}=\Se_{\tilde{\G}}/\oo$ which finishes the proof.

     
%
%
%
 \end{proof}

    \begin{lemma}
        \label{lm:perturb}
        Let $r\in\Q_{>0}$ be arbitrary. For every $M>0$ there exists $r_M\in \Q$
        such that 
        \begin{equation}\label{eq:modceil}
        \lceil r\ell\rceil=\lceil r'\ell\rceil\quad\mbox{ for every } \ell\in\Na, \ell\leq M
        \mbox{ and }r'\in(r_M,r].    
        \end{equation}
    \end{lemma}
    \begin{proof}
        Since $r\in\Q_{>0}$ one writes $r=\w/\al$, where $\w,\al\in\Na$ and $\gcd(\w,\al)=1$. For a fixed $\ell\in \mathbb{N}$ we introduce the notation $x:=\lceil \w\ell/\al \rceil$ for simplicity. Notice that if $\ell=0$ then (\ref{eq:modceil}) holds for any $r'$, so in the sequel we assume that $\ell\neq 0$.
        By the properties of the 
        function $t\mapsto \lceil t\rceil$ one knows the inequalities
        \[
        x-1+\frac{1}{\al}\le \frac{\w\ell}{\al}\le x.    
        \]  
        This implies that for any $r'\in(r-\frac{1}{\al\ell},r]$  we have
        $
        x\ge r\ell\ge r'\ell>r\ell-\frac{1}{\al}\ge x-1   
        $,
        hence $\lceil r'\ell\rceil=x=\lceil r\ell\rceil$.
        
        Thus, for a fixed $\ell$ we have constructed an interval for $r'$ such that (\ref{eq:modceil}) is satisfied. When $\ell$ is varying in $[0,M]$ we consider the intersection of these intervals 
        \[
            \bigcap_{0<\ell\le M} \left(r-\frac{1}{\al\ell},r\right]=\left(r-\frac{1}{\al M},r\right]=:(r_M,r].  
        \]
    
    \end{proof}

    \begin{lemma}\label{lem:pertN}
        Let $N\colon\Z\to\Z$, $N(\ell)=b_0\ell-\sum_{i=1}^n s_i\lceil \w_i\ell/\al_i\rceil$
        be a quasi-linear function of a representable semigroup. 
        Then for every $M\in \Na$ there exists a modification   
        $N'\colon\Z\to\Z, N'(\ell)=b_0\ell-\sum_{i=1}^{n} s_i\lceil \w'_i\ell/\al'_i\rceil$ (
        $\w'_i,\al'_i\in\Z_{>0}$, $0<\w_i'<\al'_i$ and $\gcd(\w'_i,\al'_i)=1$), such that 
        $\al'_1,\dots,\al'_n$ are pairwise relatively prime, $\gcd(\al'_i,s_i)=1$, and it satisfies  
        \[
        N'(\ell )=N(\ell) \ \ \mbox{for every} \ \ell\in \mathbb{N} \ \mbox{and} \ \ell\leq M .
        \] 
    \end{lemma}
    \begin{proof}
    
        Let $\Gamma$ be the graph corresponding to $N(\ell)$. We will say that the $i$-th block of $\Gamma$ consists of the $s_i$ number of legs with Seifert invariant $(\al_i,\w_i)$. For a fixed $M\in\Na^*$ we will construct the modification by induction on $i$. 
        
        First we describe the inductive step. Assume that the $(i-1)$-th block is already modified. This means that the new Seifert invariants $(\al'_1,\w'_1),\dots,(\al'_{i-1},\w'_{i-1})$ are constructed in a way that $\al'_t$ are pairwise relatively primes and $\gcd(\al'_t,s_t)=1$ for any $t\in\{1,\dots,i-1\}$.  
        
        Then, for a large enough $k\in\Na$, by Lemma \ref{lm:perturb} one finds a rational number $r'\in(r'_{M,i},\w_i/\al_i)$ of the form $r'=x/(k\al'_1\dots\al'_{i-1}s_i+1)$ ($x\in\Na$) satisfying $\lceil r'\ell\rceil=\lceil\w_i\ell/\al_i\rceil$ for any $\ell\leq M$. This, written as  $\w'_i/\al'_i:=r'$ (where $\gcd(\w'_i,\al'_i)=1$) gives us the perturbation. Since  $\al'_i$ is a divisor of 
        $k\al'_1\dots\al'_{i-1}s_i+1$, it is relatively prime to all the $\al'_t$ with $t\leq i-1$ and $s_i$. In this way, we get the $i$-th block of the modified graph with $s_i$ legs, all having the Seifert invariant $(\al'_i,\w'_i)$.
        
        For $i=1$ we can start by distinguishing two cases:
        
        I. \ If $\gcd(\al_1,s_1)=1$ then we do not modify them and we set $(\al'_1,\w'_1):=(\al_1,\w_1)$. 
        
        II. \ If $\gcd(\al_1,s_1)\neq 1$ then we do the same as in the inductive step. Therefore,  we get a rational number of the form $r'=x/(ks_1+1)$ ($x\in\Na$) satisfying $\lceil r'\ell\rceil=\lceil\w_1\ell/\al_1\rceil$ for any $\ell\leq M$, which provides $\w'_1/\al'_1:=r'$.
        
        By the construction one follows that $N(\ell)=N'(\ell)$ for any $\ell\leq M$.
    \end{proof}
    
 \begin{theorem}\label{charthm}
A numerical semigroup is representable if and only if it is a quotient of a flat semigroup.  
 \end{theorem}
 
 \begin{proof}
By  Lemma \ref{lem:quotientrep} and Theorem \ref{thm_flat_rep} follow that a quotient of a flat semigroup is representable. 
For the converse, let $S$ be a representable semigroup and consider one of its representative $\Gamma$ defined by the Seifert invariants $(-b_0,s_1\times(\al_1,\w_1),\dots,s_{n}\times(\al_{n},\w_{n}))$. Let $M$ be the maximum of the largest generator (in the minimal set of generators) and the Frobenius number $f_\calS$. If we apply Lemma \ref{lem:pertN} for this fixed $M$, we get a new graph $\Gamma'$ defined by the perturbed Seifert invariants $(-b_0,s_1\times(\al'_1,\w'_1),\dots,s_{n}\times(\al'_{n},\w'_{n}))$ satisfying that  $\al'_1,\dots\al'_n$ are pairwise relatively prime integers and $\gcd(\al'_i,s'_{i'})=1$. Moreover, the associated semigroup $\calS_{\Gamma'}$ coincides with $\calS$. Indeed, the identity $N(\ell)=N(\ell')$ valid for all the integers $\ell\leq f_S$ implies that  $\calS_{\Gamma'}\subset \calS$. On the other hand, $N(\ell)=N(\ell')$ for all the integers smaller than the largest generator deduces  that $\calS\subset \calS_{\Gamma'}$. Finally, Theorem \ref{flat:divisor} applied to $\Gamma'$ clarifies the statement.
 \end{proof}

\begin{example} 
Consider the non-flat semigroup $G(4,6,7,9)$. We claim that it is representable and one of its representative is given by the SSR graph with Seifert invariants $Sf=(-2, 2\times (2,1), 2\times (4,1), (5,1))$ (the calculations were performed using GAP\cite{GAP4}). Following our previous discussions, a perturbation of the Seifert invariants can be chosen as follows.

Let us first consider the two legs with $(2,1)$. 
Note that  $M=9$ and $(2,1)$ can be changed to eg. $(11,5)$ since $\frac{5}{11}\in (\frac{1}{2}-\frac{1}{18}=\frac{4}{9},\frac{1}{2})$. Similarly,  $(4,1)$ can be changed to $(13,3)$. These perturbations give us a new graph defined by the  Seifert invariants $Sf=(-2, 2\times (11,5), 2\times (13,3), (5,1))$, which has $\oo=307$ and satisfies the assumptions of Theorem \ref{flat:divisor}. Finally, we have
$$G(4,6,7,9)=\frac{1}{307}G(110,130,143).$$
\end{example}
    
\subsection{Further speculations and remarks} \label{ss:specrem}

We would like to propose a general question and emphasize some directions opened by the problem studied in this article. 

In the theory of numerical semigroups one knows by \cite{symmhalf} (see also \cite{book_nsgp}) that every numerical semigroup 
can be presented as one half of  a symmetric numerical semigroup. More generally, for a fixed $d\geq 2$, every semigroup  can be presented as one over $d$ of a symmetric numerical semigroup, cf. \cite{symmquotient}.  Note that the Example \ref{ex:G35_2} is also an example for the construction given by Rosales and Garc\'ia-S\'anchez in \cite{symmhalf}. 

As a consequence, by applying Lemma \ref{lem:quotientrep}, one follows that if we can prove that every symmetric semigroup is representable, then 
every numerical semigroup is representable. This approach naturally poses the following question. 
\begin{question}
 Is there a symmetric numerical semigroup which is not representable?
\end{question}

By the knowledge of the authors, there is no good understanding how the symmetric property of semigroups incarnates on the level of the representatives, or how does it fit in the `flat' classification theme of Raczunas and Chrz\c astowski-Wachtel. 

For example, one can construct representable 
 semigroups which are not symmetric, but they have a numerically Gorenstein representative:
 \begin{example}[Non-symmetric numerically Gorenstein case \cite{stronglyflat}]
    Let $\G$ be the graph 
    defined by the Seifert invariants 
    $$Sf=(-2,(2,1),(2,1),(3,1),(3,1),(7,1),(7,1),(84,1)).$$
    Then  $Z_K=(86,43,43,29,29,13,13,2)$,
    hence $\G$ is numerically Gorenstein. In addition,
    one can compute that $\gamma=85$, $1/|e|=28$
    and $\check{s}=28$ (cf. (\ref{eq:frob}) ), thus the Frobenius number equals $f_{\Se_\G}=85$.
    In addition $N(6)=N(85-6)=-1$ which implies $6,f_{\Se_\G}-6\notin\Se_\G$,
    hence it is not symmetric.
\end{example}
 One the other hand, using the proof of Theorem \ref{thm_ngor0_sym_and_frob} one can construct numerical semigroups
 which  are symmetric, but not flat, as shown by the next example. Therefore, the set of symmetric semigroups is rather bigger than the set  of the flat ones. 
 \begin{example}[Symmetric but not flat]
    Let $\G$ be the graph defined by the Seifert invariants
    \[
     Sf=(-2,(35,13),(35,13),(21,13),(21,13)).   
    \]
    One can verify eg. with \cite{GAP4} that $\Se_\G=G(8,21,35)$. Hence
    $\Se_\G$ is not a flat but an almost flat semigroup. However, it is symmetric. 
 \end{example}

\end{document}